\documentclass[12 pt,reqno]{amsart}
\usepackage{amsmath}
\usepackage{amsfonts}
\usepackage{amssymb}
\usepackage{amsthm}
\usepackage{amscd}
\usepackage{a4wide}
\usepackage{enumerate}
\usepackage{pifont}
\usepackage{dsfont} % i added for blackboard bold 1's
\usepackage{xcolor}
\usepackage{amsmath}
\usepackage{tikz}
\usepackage{tikz-cd}
\usepackage{booktabs}
\usetikzlibrary{shapes.geometric}
\usepackage{amscd}
\usepackage{mathabx}
\numberwithin{equation}{section}
\usepackage[original]{imakeidx}
\makeindex % will get the same name as the main file
\usepackage{mathrsfs}
\usepackage{graphics}
\usepackage{eucal}
\usepackage{mathrsfs}
\usepackage{mdwlist}
\usepackage{color}
\usepackage[pdftex,bookmarks,colorlinks,breaklinks]{hyperref}  % PDF hyperlinks, with coloured links
\usepackage{mdwlist}
\usepackage{environ}
\hypersetup{linkcolor=red,citecolor=blue,filecolor=dullmagenta,urlcolor=darkblue} 
\usepackage{enumitem}
\usepackage{multicol}
\setenumerate[1]{label=\thesection.\arabic*.}
\newif\ifhinting\hintingtrue % hints are on by defautl% coloured links
\usepackage{fancyhdr}
\fancypagestyle{plain}{ %
  \fancyhf{} % remove everything
}
\newtheorem {theorem}    {Theorem}[section]
\newtheorem {lemma}      [theorem]    {Lemma}

\newtheorem {corollary}  [theorem]    {Corollary}
\newtheorem {proposition}[theorem]    {Proposition}

\theoremstyle{definition}
\newtheorem {definition} [theorem]    {Definition}

\newtheorem {remark}    [theorem]    {Remark}

\newcounter{AbcT}

\numberwithin{equation}{section}

\newcommand {\N} {{\mathbb N}}
\newcommand {\R} {{\mathbb R}}

\newcommand{\Q}{{\mathbb Q}}

\newcommand {\T} {{\mathbb T}}
\newcommand {\Z} {{\mathbb Z}}

\newcommand{\e}{\varepsilon}

\newcommand{\IGNORE}[1]{}
\newcommand{\tmX}{{\tilde{\mathcal{X}}}}

\newcommand{\cB}{\mathbb{B}}

\DeclareMathOperator{\supp}{supp}

\DeclareMathOperator{\Span}{Span}

 \DeclareMathOperator{\SL}{SL}
\DeclareMathOperator{\Y}{\mathcal{Y}}

\newcommand{\diag}{\operatorname{diag}}
\newcommand\Lie{\operatorname{Lie}}

\newcommand{\bbR}{\mathbb{R}}

\newcommand{\ccG}{\mathcal{G}}

\newcommand{\aaa}{c}

\newcommand{\ind}{\protect\raisebox{2pt}{$\chi$}}

\newcommand{\cum}{\operatorname{Cum}}
\newcommand{\Pp}{\mathcal{P}}
\newcommand{\X}{\mathcal{X}}
\newcommand{\Mat}{\mathrm{M}_{m \times n}(\bbR)}
\newcommand{\Matt}{\mathrm{M}_{m \times n}([0,1])}

\newcommand{\Liou}{\text{Liou}}
\newcommand{\Lip}{\text{Lip}}

\begin{document}
\title{Effective multi-equidistribution for translates of unipotent flows and Central limit theorems in inhomogeneous Diophantine approximation}

\author{Gaurav Aggarwal}
\address{\textbf{Gaurav Aggarwal} \\
Institut f\"ur Mathematik, Universit\"at Z\"urich, 8057 Z\"urich, Switzerland}
\email{gaurav.aggarwal@math.uzh.ch}

\author{Sourav Das}
\address{\textbf{Sourav Das} \\
School of Mathematics,
Tata Institute of Fundamental Research, Mumbai, India 400005}
\email{iamsouravdas1@gmail.com, sourav@math.tifr.res.in}

\author{Anish Ghosh}
\address{\textbf{Anish Ghosh} \\
School of Mathematics,
Tata Institute of Fundamental Research, Mumbai, India 400005}
\email{ghosh@math.tifr.res.in}

\thanks{ G.~A. gratefully acknowledges support from the Swiss National Science Foundation (grant 200020-212617). All three authors acknowledge support from the Department of Atomic Energy, Government of India (project 12-R\&D-TFR-5.01-0500). A.\ G.\ gratefully acknowledges support from a grant from the Infosys foundation, and a J. C. Bose grant of the ANRF}

\begin{abstract}
In this paper, we prove a central limit theorem for inhomogeneous Diophantine approximation with a fixed shift, provided the shift is non-Liouville. This generalizes earlier work of Dolgopyat, Fayad, and Vinogradov~\cite{DFV}.
This is achieved by translating the problem to one involving flows on homogeneous spaces. In this latter setting, we establish an effective multi-equidistribution result for diagonal translates of unipotent flows. This result is obtained by combining a recent result of Kim~\cite{Kim2024} with the height function construction of Shi~\cite{Shi20}. The central limit theorem is then deduced using the method of Bj\"orklund and Gorodnik~\cite{BG}.
\end{abstract}


%\subjclass[2020]{11J61, 11J83, 11K60, 37D40, 37A17, 22E40} \keywords{
 % Dynamical systems, Homogeneous flows}

\maketitle

\tableofcontents

\section{Introduction}\label{sec:intro}
This paper lies at the intersection of Diophantine approximation, homogeneous dynamics, and probability theory. On the dynamical side, we study diagonal translates of unipotent flows on the space of affine unimodular lattices and establish an \emph{effective multi-equidistribution} result (Theorem~\ref{thm: main thm equidistribution}). On the Diophantine and probabilistic side, we use this result to deduce a central limit theorem for Diophantine approximates (Theorem~\ref{thm: main thm}).

\subsection{Diophantine approximation}
Diophantine approximation studies how well real numbers can be approximated by rationals, typically through expressions of the form $|q\theta - p|$. Equivalently, this concerns the distribution of $q\theta \,\,(\bmod\, 1)$ near zero. A natural refinement is to study approximation near an arbitrary fixed shift, that is, how close $q\theta \,\,(\bmod \, 1)$ can get to a fixed point $\xi$, that is, expressions of the form $|q\theta - p - \xi|$. The study of such objects goes by the name of \emph{inhomogeneous} Diophantine approximation and is concerned, more generally, with systems of affine forms. We will study such forms with two additional complications. Firstly, we will consider forms where the `shift', namely $\xi$  is fixed. Since this naturally involves averaging over a smaller family, it is a more difficult problem. Secondly, we will adorn the relevant Diophantine inequalities with \emph{weights}, which significantly complicates the problem compared to the equal-weight case due to unequal approximation rates across different coordinates. We dive into specifics without further ado. For more details on weighted and inhomogeneous Diophantine approximation, the reader is referred to \cite{AG24singular,aggarwal2025,CGGMS}  and the references therein.\\

Fix $m,n \ge 1$. To each pair $(\theta,\xi)\in \mathrm{M}_{m\times n}(\mathbb R)\times\mathbb R^m$, associate the system of affine linear forms
\[
L^{(i)}_{\theta,\xi}(y)
=
\sum_{j=1}^n \theta_{ij} y_j + \xi_i,
\qquad y=(y_1,\dots,y_n)\in\mathbb R^n,\quad \theta=(\theta_{ij}),\quad i=1,\dots,m.
\]
The central question is how well $\theta \cdot \mathbf q \,\, (\bmod \,1)$ approximates $\xi$ in $\mathbb R^m$, or equivalently, to find $(\mathbf p,\mathbf q)\in \mathbb Z^m\times(\mathbb Z^n\setminus\{0\})$ such that the quantities $L^{(i)}_{\theta,\xi}(\mathbf q)+p_i$ are simultaneously small for all $i$. This is typically expressed by requiring $|L^{(i)}_{\theta,\xi}(\mathbf q)+p_i|$ to be uniformly small across all coordinates.

A more refined question is to seek different approximation rates in different directions, leading to the weighted theory. To make this precise, fix the sup norm $\|\cdot\|$ on $\mathbb R^n$, parameters $\vartheta_1,\dots,\vartheta_m>0$, and weights $w_1,\dots,w_m>0$ satisfying $w_1+\cdots+w_m=n$. The problem is then to study integer solutions $(\mathbf p,\mathbf q)$ of the inequalities
\begin{equation}\label{equ:Dio}
\left|L^{(i)}_{\theta,\xi}(\mathbf q)+p_i\right|
\le
\frac{\vartheta_i}{\|\mathbf q\|^{w_i}},
\qquad i=1,\dots,m,
\end{equation}
where $(\mathbf p,\mathbf q)\in \mathbb Z^m\times(\mathbb Z^n\setminus\{0\})$.

A classical theorem of Schmidt \cite[Prop. 3]{Sch}, extending Khintchine's theorem to the inhomogeneous setting with fixed shift, implies that for any fixed $\xi\in\mathbb R^m$ the system \eqref{equ:Dio} admits infinitely many solutions $(\mathbf p,\mathbf q)$ for Lebesgue almost every $\theta\in\Mat$. This naturally leads to the problem of understanding the asymptotic distribution of such solutions as $\|\mathbf q\|$ grows.

To study this question, we introduce the counting function
\begin{equation}\label{equ:count_func}
\Delta_T^\xi(\theta)
:=
\#\Bigl\{(\mathbf p,\mathbf q)\in
\mathbb Z^m\times(\mathbb Z^n\setminus\{0\}) :
1\le\|\mathbf q\|\le T
\ \text{and}\ \eqref{equ:Dio}\ \text{holds}
\Bigr\}.
\end{equation}

Schmidt \cite[Prop.~3]{Sch} proved for any shift $\xi$ and for Lebesgue almost every
$\theta\in M_{m\times n}([0,1])$,
\[
\Delta_T^\xi(\theta)
=
C_{m,n}\log T
+
O_{\theta,\varepsilon}\big((\log T)^{1/2+\varepsilon}\big)
\quad \text{for all }\varepsilon>0,
\]
where $C_{m,n}$ is a constant depending only on $m$, $n$ and the chosen norm on $\R^n$.\\

Schmidt's result can be viewed as an analogue of the law of large numbers for $\Delta_T^\xi(\theta)$. Hence, it is natural to seek deeper probabilistic information, for instance, about the fluctuations around the mean of $\Delta_T^\xi(\theta)$. A central limit theorem describing these fluctuations was first established by Dolgopyat, Fayad and Vinogradov in \cite{DFV}, where the authors considered the case of equal weights. They proved the result in the homogeneous case $\xi=0$, and also for generic, i.e., Lebesgue almost every $\xi$. Subsequently, Bj\"{o}rklund and Gorodnik \cite{BG} extended the central limit theorem to unequal weights in the homogeneous case $\xi=0$. More recently, the first and third named authors \cite{AG23} established a central limit theorem in the unequal weight setting for rational shifts $\xi \in \Q^m$, as well as in the setting where the shift $\xi$ also varies. The theme of central limit theorems in Diophantine approximation has an older vintage, beginning with works of Philipp \cite{Philipp}, Samur \cite{Samur}, and others, and has seen considerable activity of late. For more recent results, we refer the reader to \cite{BG17,BG23,BG2,Das,AG25,AGh25,BFG26}.\\

The purpose of the present paper is to extend the fixed shift result  \cite[Theorem 1.4]{DFV} to a much larger and \emph{explicit} class of inhomogeneous shifts.

  Recall that a vector $\xi\in\mathbb R^\ell$ is called a \emph{Liouville vector} if for every $E>0$ there exists infinitely many $(\mathbf p,q)\in\mathbb Z^\ell\times\mathbb N$ such that
\[
\left\|\xi-\frac{\mathbf p}{q}\right\|
<
\frac{1}{q^{E}},
\]
where $\|\cdot\|$ denotes the sup norm on $\mathbb R^\ell$. We denote by $\mathrm{Liou}(\mathbb R^\ell)$ the set of all Liouville vectors in $\mathbb R^\ell$. In particular, we prove the following.

\begin{theorem}\label{thm: main thm}
Let $m\ge2$. Suppose $\xi\notin\mathrm{Liou}(\mathbb R^m)$, i.e.,\ $\xi$ be a non-Liouville vector. Then for any $\eta\in\mathbb R$,
\[
\left|
\left\{
\theta\in  \Matt:
\frac{\Delta_T^\xi(\theta)-C_{m,n}\log T}{(\log T)^{1/2}}
<
\eta
\right\}
\right|
\longrightarrow
\mathcal N_{\sigma_{m,n}}(\eta)
\quad\text{as }T\to\infty,
\]
where $|\cdot|$ denotes Lebesgue measure on $ \Matt$,
\[
\mathcal N_\sigma(\eta)
:=
\frac{1}{\sqrt{2\pi\sigma}}
\int_{-\infty}^{\eta} e^{-x^2/(2\sigma)}\,dx
\]
is the normal distribution with variance $\sigma$, and
\[
C_{m,n}
=
\sigma_{m,n}^2
=
2^m \vartheta_1\cdots\vartheta_m\,\omega_n,
\qquad
\omega_n:=\int_{\mathbb S^{n-1}} d\bar z .
\]
\end{theorem}
\begin{remark}
We stress that even in the equal-weight setting, Theorem~\ref{thm: main thm} significantly strengthens~  \cite[Theorem 1.4]{DFV}. Indeed, the set of vectors $\xi \notin \mathrm{Liou}(\mathbb R^m)$ has full Lebesgue measure, thereby recovering their result and furthermore strengthening it since the exceptional set $\mathrm{Liou}(\mathbb R^m)$ has Hausdorff dimension zero. Thus, Theorem \ref{thm: main thm} applies to a substantially larger class of shifts. Moreover, since the condition $\xi \notin \mathrm{Liou}(\mathbb R^m)$ is explicit, it allows one to produce concrete examples where Theorem \ref{thm: main thm} applies. For instance, Theorem~\ref{thm: main thm} holds for vectors such as
\[
(e,\dots,e), \quad (\pi,\dots,\pi), \quad (\sqrt{2},\dots,\sqrt{2}),
\]
and, more generally, whenever at least one coordinate is algebraic (since every algebraic number is non-Liouville by Liouville’s theorem).
\end{remark}

\begin{remark}
It is natural to enquire, however, if the Diophantine condition we impose is necessary. We already know this not to be the case because a CLT with fixed \emph{rational} shift can be found in \cite{AG23}. What is clear is that Theorem \ref{thm: main thm equidistribution} below needs this condition. In light of W. M. Schmidt's almost sure asymptotic counting theorem which holds with no conditions, and the fixed rational shift result from \cite{AG23}, we pose the following question with cautious optimism.\\

\textbf{Question}: Prove an analogue of Theorem \ref{thm: main thm} for arbitrary fixed shifts.\\

We remark that the variance is expected to depend on arithmetic properties of the shift and may differ for rational or near-rational vectors.
\end{remark}

\subsection{Homogeneous dynamics}

We now state the effective multi-equidistribution for diagonal translates of unipotent flows on the space of affine unimodular lattices that will be key ingredient in the proof of Theorem~\ref{thm: main thm}. To do the same, we will need the following notation.

Let $d := m+n$, and define
\begin{align*}
    \tilde G := \SL_d(\mathbb{R}) \ltimes \mathbb{R}^d,
\qquad
    \tilde \Gamma := \SL_d(\mathbb{Z}) \ltimes \mathbb{Z}^d.
\end{align*}
Let $\tilde{\X} := \tilde G / \tilde \Gamma$, and denote by $\mu_{\tilde{\X}}$ the unique $\tilde G$-invariant probability measure on $\tilde{\X}$.

Fix weights $w_1,\dots,w_{m+n} > 0$ satisfying
\[
\sum_{i=1}^{m} w_i = \sum_{j=1}^{n} w_{m+j},
\]
and define, for $t \in \mathbb{R}$ and $\theta \in \Mat$,
\begin{equation}\label{equ:b_t}
    a_t := \diag \bigl(e^{w_1 t},\dots,e^{w_m t},e^{-w_{m+1}t},\dots,e^{-w_{m+n}t}\bigr), \quad u(\theta) :=
\begin{pmatrix}
    I_m & \theta \\
    0 & I_n
\end{pmatrix}.
\end{equation}
Throughout this paper, we will view $u(\theta)$ and $a_t$ as elements of $\tilde G$ via the embedding $g \mapsto [g,0]$.

\medskip

The main dynamical result of the paper is the following theorem.

\begin{theorem}\label{thm: main thm equidistribution}
Fix $y = [g,gv]\tilde{\Gamma} \in \tilde{\X}$ with $v \notin \mathrm{Liou}(\mathbb{R}^d)$, $\ell \in \mathbb{N}$, and a compact subset $V \subset \Mat$. Then there exists a constant $\delta = \delta(y,\ell) > 0$ and $k \in \mathbb{N}$ such that the following holds.

For any $F_0 \in C_c^\infty(\Mat)$ with $\mathrm{supp}(F_0) \subset V$, and for all $F_1,\dots,F_\ell \in C_c^\infty(\tilde{\X})$, and all $t_1,\dots,t_\ell > 0$, we have
\begin{align*}
\int_{\Mat} F_0(\theta) \prod_{i=1}^\ell F_i(a_{t_i}u(\theta)y)\, d\theta
&=
\left( \int_{\Mat} F_0(\theta)\, d\theta \right)
\left( \prod_{i=1}^\ell \int_{\tilde{\X}} F_i\, d\mu_{\tilde{\X}} \right) \\
&\quad + O_{y,\ell,V}\!\left(
e^{-\delta D(t_1,\dots,t_\ell)}
\|F_0\|_{C^1}
\prod_{i=1}^{\ell} \mathcal{S}_k(F_i)
\right),
\end{align*}
where the Sobolev norm $\mathcal{S}_{k}(\cdot)$ is defined as in~\eqref{eq:sobolev norm} and 
\begin{align}
    \label{eq:def D}
D(t_1,\dots,t_\ell) := \min\{t_i,\ |t_i - t_j| : 1 \le i \ne j \le \ell\}.
\end{align}
\end{theorem}

We refer the reader to \cite{Strom-2015, Strom-Vishe, Prin, Browning-Vinogradov, Kim2024} and the references therein for earlier results on this very active area. The condition on $v$ is natural; indeed it is well known that even effective single equidistribution theorems in the present context are characterized by Diophantine properties of the shift affecting the rate of equidistribution. We refer the reader to the references above, especially \cite{Kim2024}. As mentioned above, it is satisfying that Theorem \ref{thm: main thm equidistribution} applies to an explicit and very large set of vectors, i.e. the complement of the Liouville vectors, a $0$ Hausdorff dimension set.    

\subsection*{Strategy of the proof}

The main new technical contribution of the paper lies in the proof of Theorem~\ref{thm: main thm equidistribution}. To illustrate the key idea underlying the proof, we restrict to the simplest case $m=n=1$ and $d=2$.

The strategy of deducing multi-equidistribution from single equidistribution on homogeneous spaces is well known, originating in the works~\cite{KSW, BG,BGeff}. The argument proceeds by induction: for the case of the space of affine lattices, the base case, namely, single equidistribution, follows from Kim~\cite{Kim2024} (Theorem~\ref{thm:kim}), while the inductive step reduces to establishing a suitable form of \emph{double equidistribution}. More precisely, one needs to show that for any $s>r>0$, the measure supported on $a_{r+s}u(\theta)y$, as $\theta$ varies with respect to a distribution $f(\theta)\,d\theta$, is close to the Haar measure on $\tilde{\X}$, with an error of the form
\[
e^{-\delta r}\|f\|_{\Lip} + e^{-\delta s} \|f\|_{C^0}.
\]
Once this estimate is established, the passage to multi-equidistribution is standard.\\

\noindent The proof of double equidistribution begins by decomposing the interval $[0,1]$ into subintervals of length $e^{-2s}$. Writing $\theta = e^{-2s}(j + \theta')$ with $\theta' \in [0,1]$, we obtain
\begin{align} \nonumber
\int_{0}^1 f_0(\theta)\, f_1(a_{r+s}u(\theta)y)\, d\theta
= \frac{1}{\lfloor e^{2s}\rfloor} \sum_{j=0}^{\lfloor e^{2s}\rfloor} \int_{0}^{1} 
f_0(e^{-2s}(j+\theta'))\, f_1(a_r u(\theta')a_su(e^{-2s}j)y)\,d\theta' 
+ O(e^{-2s}).
\end{align}
For each $j$, the function $\theta' \mapsto f_0(e^{-2s}(j+\theta'))$ is nearly constant, and can be approximated by a constant up to an error $O(e^{-2s})$. Thus, it suffices to analyze integrals of the form
\begin{align}
\label{eq:intro strategy 2}
\int_{0}^{1} f_1(a_r u(\theta)a_su(e^{-2s}j)y)\,d\theta,
\end{align}
to which we apply single equidistribution and then sum over all $j$.

The key point to note here is that the rate of convergence in \eqref{eq:intro strategy 2} depends on the base point $a_su(e^{-2s}j)y$. To obtain an overall exponential rate in the sum, one must show that for most $j$, the convergence in~\eqref{eq:intro strategy 2} is exponentially fast.

In the works~\cite{KSW, BG}, which treat the homogeneous space of unimodular lattices, the rate of convergence depends only on the distance to the cusp. In that setting, one fixes a compact set (depending on $r$) on which convergence is uniformly exponentially fast, and shows—using non-divergence estimates for unipotent flows—that, with high probability, the points $a_su(e^{-2s}j)y$ lie in this compact set. The exceptional set has measure $\ll e^{-\delta s}+e^{-\delta r}$, which yields the desired estimate.\\

In the present setting, the situation is significantly more delicate. The rate of convergence depends not only on the distance to the cusp, but also on the distance to infinitely many $\SL_2(\R)$-invariant homogeneous subspaces embedded in $\tilde{\X}$. The main new technical ingredient of this paper is to show that, with high probability, the points $a_su(e^{-2s}j)y$ stay away from both the cusp and these subspaces. This ensures uniform exponential convergence to Haar measure for most points, thereby yielding double equidistribution.

To make this precise, we use single equidistribution (Theorem~\ref{thm:kim}) to note that the condition of being away from both the cusp and the invariant subspaces can be expressed in terms of the shift and the norm of a representative $y=[g,gv]\tilde{\Gamma}$ as
\[
\|g\| \leq \zeta(v,e^{wr})^\delta.
\]
We refer to Section~\ref{subsec:gen} for the definition of $\zeta(\cdot)$ and further details.

In Section~\ref{subsec:gen}, we show that violating this condition is equivalent to the existence of a short vector in a thin cylinder for the lattice
\[
\begin{pmatrix}
    g & gv \\ & 1
\end{pmatrix}\Z^{d+1}.
\]
Using this observation, the problem reduces to showing that for any $y=[g,gv]\tilde{\Gamma}\in \tilde{\X}$, for most $\theta$, the lattice
\[
\begin{pmatrix}
    a_t u(\theta) \\ & 1
\end{pmatrix}
\begin{pmatrix}
    g & gv \\ & 1
\end{pmatrix}\Z^{d+1}
\]
does not contain a vector in a thin cylinder. This is precisely the content of Section~\ref{sec:Shi}, where the argument relies on the height function machinery developed by Shi~\cite{Shi20}. 

Combining the above steps, we obtain double equidistribution, thereby completing the proof of Theorem~\ref{thm: main thm equidistribution}.

Finally, Theorem~\ref{thm: main thm} follows from Theorem~\ref{thm: main thm equidistribution} using the approach of~\cite{BG}. This involves reducing the problem to a central limit theorem for Birkhoff sums of a Siegel transform and then applying effective equidistribution to control the cumulants (see Section~\ref{sec:mainproof} for details).

\subsection*{Acknowledgements}
Part of this work was carried out while the first and third authors were visiting CIRM, Luminy. The hospitality of the Institute is gratefully acknowledged. The first author thanks Alexander Gorodnik for various helpful discussions during the preparation of this paper.

\section{Notation I} \label{sec:notation 1}

\subsection{Sobolev norm}
For $Y \in \mathrm{Lie}(\tilde{G})$, let ${D}_Y$ denote the first-order differential operator on 
$C_c^{\infty}( \tilde{\mathcal{X}})$ defined by
\[
{D}_Y(f)(y) := \left.\frac{d}{dt} \, f(\exp(tY)y)\right|_{t=0}.
\]  
\noindent Fix an ordered basis $\{Y_1,\dots,Y_r\}$ of $\Lie(\tilde{G})$. For a monomial $Z = Y_1^{\ell_1} \cdots Y_r^{\ell_r}$, define the corresponding differential operator
\[
D_Z := D_{Y_1}^{\ell_1} \cdots D_{Y_r}^{\ell_r},
\]
and its degree by
\[
\deg(Z) := \ell_1 + \cdots + \ell_r.
\]
We denote the uniform norm on $C_c( \tilde{\mathcal{X}})$ by $\|\cdot\|_{C^0}$ and for $k\geq 1,$ we define the $C^k$ norm on $C_c^{\infty}( \tilde{\mathcal{X}})$ by
\[\|f\|_{C^k}:= \sum_{\deg(Z) \leq k} \big\| {D}_Z f \big\|_{C^0} \]

\noindent For $k \geq 1$ and $f \in C_c^\infty( \tilde{\mathcal{X}})$, we define the $L^2$-Sobolev norm
\begin{align}
    \label{eq:sobolev norm}
    \mathcal{S}_{k}(f) 
:= \max\left\{ \left( \sum_{\deg(Z) \leq k} \int_{ \tilde{\mathcal{X}}} \big|\text{ht}(y)^k ({D}_Z f)(y) \big|^2 \, d\mu_{ \tilde{\mathcal{X}}}(y) \right)^{1/2}, \,\, \|f\|_{C^k}  \right\},
\end{align}
where for $y \in \tmX$,
\[
\text{ht}(y):=\sup \left\{\|gw\|^{-1}: y=[g,gv]\tilde{\Gamma},\,\, w \in \mathbb Z^d \setminus \{0\}\right\}.
\]
 
\begin{remark}\label{rem:sobnorm_pro}  
In view of \cite[Equation (3.11)]{EMV}, there exists a constant $\kappa=\kappa(m,n,k)$ (which also depends on the fixed choice of weights $w_1,\dots,w_{m+n}$) such that
\[\mathcal{S}_k(f \circ a_t) \ll e^{\kappa t} \mathcal{S}_k(f), \quad \forall \,t \in \mathbb R \text{ and } f \in C_c^{\infty}(\tilde{\X}),\]
     where the implied constant is independent of both $t $ and $f.$
\end{remark}
  
\subsection{Weighted Lipschitz norm $ \|\cdot\|_{w,t}$} \label{subsec:weighted norm}
For $t \in \R$, we define 
   $$
   b_t:= \begin{pmatrix}
       e^{w_1 t}\\ & \ddots \\&& e^{w_m t}
   \end{pmatrix} , \quad c_t:= \begin{pmatrix}
       e^{w_{n+1} t}\\ & \ddots \\&& e^{w_{m+n} t}
   \end{pmatrix}.
   $$
   It is clear that for any matrix $\theta \in \Mat$, we have
   \begin{align}
       \label{eq:a1}
       a_tu(\theta)a_t^{-1} = u(b_t\theta c_t).
   \end{align}

For a function $f \in C_c^{\infty}(\Mat)$ and $t>0$, we define
\begin{align}
\label{eq:def weighted norm}
    \|f\|_{w,t}:= \sup\left\{|f(\eta+ b_t^{-1} \theta c_t^{-1})- f(\eta)|: \theta \in \mathrm{M}_{m \times n}([-1/2,1/2]), \eta \in \Mat\right\}.
\end{align}
The following lemma gives a bound on $ \|f\|_{w,t}$.
\begin{lemma}
    \label{lem: C^k norm of restricted function}
   For each $ r\geq 1$, there exists a constant $c_r>0$ such that the following holds. Suppose $F_0 \in C_c^\infty(\Mat)$ and $F_1, \ldots, F_{r} \in C_c^\infty(\tmX)$. Suppose $0<t_1\leq \cdots \leq t_r\leq s$ and $y \in \tmX$. Then the norm $\|\cdot\|_{w,s}$ of the function 
   \begin{align}
       \label{eq: aaa 1}
   \theta \mapsto F_0(\theta) F_1({a}_{t_1} u(\theta)y) \cdots F_r({a}_{t_r} u(\theta)y),
   \end{align}
  is less than $c_r e^{-4 w (s-t_{r})} \|F_0\|_{C^1} \cdots \|F_r\|_{C^1}$, where
  \begin{align}\label{eq: def w}
     w := \frac{1}{2} \min_{1 \leq i \leq d} w_i.
 \end{align}
\end{lemma}
\begin{proof}
 Fix $r \geq 1$. For $1\leq i \leq r$, define $\tilde{F}_i: \Mat \rightarrow \R$ as $$\tilde{F}_i(\theta)= F_i(a_{t_i} u(\theta)y).$$
   Clearly
   \begin{align}
    \nonumber
       \|\tilde{F}_i\|_{w,s} &=  \sup\left\{\left|\tilde{F}_i\left(\eta+ b_s^{-1} \theta c_s^{-1}\right)- \tilde{F}_i(\eta)\right|: \theta \in \mathrm{M}_{m \times n}\left([-1/2,1/2]\right), \eta \in \Mat \right\}  \\
       &= \sup\left\{\left|F_i\left(a_{t_i} u(\eta+ b_s^{-1} \theta c_s^{-1})y\right)- F_i\left(a_{t_i} u(\eta)y\right)\right|: \theta \in \mathrm{M}_{m \times n}\left([-1/2,1/2]\right), \eta \in \Mat\right\} \nonumber\\
       &= \sup\left\{\left|F_i\left( u(b_{s-t_{i}}^{-1} \theta c_{s-t_{i}}^{-1} )a_{t_i} u(\eta)y\right)- F_i(a_{t_i} u(\eta)y)\right|: \theta \in \mathrm{M}_{m \times n}([-1/2,1/2]), \eta \in \Mat\right\} \nonumber \\
       &\ll e^{-4w(s-t_i)} \|F_i\|_{C^1} \leq e^{-4w(s-t_r)} \|F_i\|_{C^1}. \nonumber
   \end{align}
   Now the $\|\cdot\|_{w,s}$ norm of the function \eqref{eq: aaa 1} equals
    \begin{align*}
        \|F_0 \cdot \tilde{F}_1 \cdots \tilde{F}_r \|_{w,s} &\leq \|F_0\|_{w,s} \prod_{i=1}^{r}\|\tilde{F}_i\|_{C^0}  +  \|F_0\|_{C^0} \sum_{i=1}^r \|\tilde{F}_i\|_{w,s} \left( \prod_{j \neq i} \|\tilde{F}_j\|_{C^0} \right) \\
        &\ll e^{-4w(s-t_r)} \|F_0\|_{C^1} \prod_{i=1}^{r}\|\tilde{F}_i\|_{C^0} +  \|F_0\|_{C^0} \sum_{i=1}^r e^{-4w(s-t_r)}\|{F}_i\|_{C^1} \left( \prod_{j \neq i} \|{F}_j\|_{C^0} \right) \\
        &\ll  e^{-4w(s-t_i)} \|F_0\|_{C^1} \cdots \|F_r\|_{C^1},
    \end{align*}
    where the implied constant depends only on $r$. Hence proved.
\end{proof}

\section{Vectors in a thin cylinder}\label{sec:Shi}
The main aim of this subsection is to prove the following proposition.
\begin{proposition}
\label{main prop}
Suppose $y=[g,gv]\tilde{\Gamma}\in \tilde{\X}$ with $v\notin \Liou(\R^d)$, and let $L>0$. Then there exist constants $\beta_0,\beta_1,\beta_2>0$ (depending on $y$ and $L$) such that for all $0<A<1<B$, the set $S(L,A,B)$ of all $\theta \in \mathrm{M}_{m \times n}([-L,L])$ for which the lattice
\[
\begin{pmatrix}
    a_t u(\theta) & \\ & 1
\end{pmatrix}
\begin{pmatrix}
    g & gv \\ &1
\end{pmatrix}\Z^{d+1}
\]
contains a nonzero vector in the region
\begin{align}
\label{eq:def R}
\mathcal{B}(A,B)= \{(x,y) \in \R^{d} \times \R : \|x\|\leq A, \ |y| \leq B \}
\end{align}
satisfies
\[
 |S(L,A,B)| \leq \beta_0 A^{\beta_1} B^{\beta_2}.
\]
\end{proposition}

\subsection{Height function}
We define \[\Y:= \SL_{d+1}(\R)/\SL_{d+1}(\Z),\] which we identify with the space of unimodular lattices in $\R^{d+1}$ via the map
$$
g \, \SL_{d+1}(\Z) \mapsto g \Z^{d+1}.
$$
Throughout this section, we regard $\SL_d(\R)$ as a subgroup of $\SL_{d+1}(\R)$ via the embedding
$$
g \mapsto \begin{pmatrix}
    g  \\ & 1
\end{pmatrix}.
$$

We now recall the construction of the height function from \cite[\S 4.1]{Shi20} for the action of $\SL_d(\R)$ on $\Y$.

\noindent Let
\[
\mathfrak{a} := \left\{ \begin{pmatrix}
    t_1 \\ & \ddots \\ && t_d
\end{pmatrix} : \sum_{i=1}^d t_i = 0 \right\}
\]
be the Lie algebra of the diagonal subgroup \[A:=\left\{\diag(e^{t_1},\dots,e^{t_d}):\sum_{i=1}^{d}t_i=0\right\} \subset \SL_d(\R).\] 
Define linear functionals $\omega_i$ on $\mathfrak{a}$ by
$$
\omega_i \left( \begin{pmatrix}
    t_1 \\ & \ddots \\ && t_d
\end{pmatrix} \right)  = t_i.
$$
\noindent Then the root system of $\mathfrak{a}$ is
\[
\Phi = \{ \omega_i - \omega_j : i \neq j \}, \qquad
\Phi^+ = \{ \omega_i - \omega_j : i < j \}.
\]

 Let $\{e_i:i=1,\dots,d+1\}$  denote the standard basis of $\mathbb R^{d+1}.$ Consider the natural representation of $\SL_d(\R)$ on
$\R^{d+1} = \R^d \oplus \R e_{d+1}$, which induces the decomposition
\[
\bigwedge^k \R^{d+1}
=
\bigwedge^k \R^d
\;\oplus\;
\left(\bigwedge^{k-1} \R^d \wedge e_{d+1}\right).
\]
Each $\bigwedge^k \R^d$ is irreducible with highest weight
\[
\eta_k = \omega_1 + \cdots + \omega_k.
\]

\noindent Thus, the set of highest weights appearing in $\bigwedge^* \R^{d+1}:=\oplus_{0 \leq i \leq d+1} \bigwedge^i \R^{d+1}$ is
\[
P^+ = \{0, \eta_1, \ldots, \eta_{d-1}\}.
\]
\noindent For each $\eta \in P^+$, let $\pi_{\eta}$ denote the projection from $\bigwedge^* \R^{d+1}$ onto the sum of subrepresentations with highest weight $\eta$. Define
$$
\delta_\eta:= \eta (\log(a_1)), \quad \delta_i:= (d+1-i)i.
$$
Note that for each $\eta \in P^+ \setminus \{0\}$, we have $\delta_\eta=\eta(\log(a_1))>0$. 

  Let $\e>0$ and $0<i<d+1$. For every $v \in \bigwedge^i \R^{d+1}$, define
\begin{align}
    \label{eq:def phi}
    \varphi_\e(v):= \begin{cases}
       \min_{\eta \in P^+ \setminus 0} \e^{\frac{\delta_i}{\delta_\eta}} \|\pi_\eta(v)\|^{\frac{-1}{\delta_\eta}} & \text{ if } \|\pi_0(v)\| \leq \e^{\delta_i} \\
        0 &\text{otherwise,}
    \end{cases}
\end{align}
\noindent where we extend the norm on $\R^{d+1}$ naturally to an $L^{\infty}$ norm on the exterior algebra.

Finally, define $\alpha_\e: \Y \rightarrow [0,\infty]$ by
\begin{align}
    \label{eq:def alpha}
    \alpha_\e(y):= \max \varphi_\e(v),
\end{align}
where the maximum is taken over all the non-zero $y$-integral monomials $v \in \bigwedge^i \R^{d+1}$ with $0 < i < d+1$. Recall that a vector $v \in \bigwedge^i \R^{d+1}$ is called a $y$-integral monomial if it can be written in the form $v = v_1 \wedge \cdots \wedge v_i$ with $v_1, \ldots, v_i \in y$, where $y$ is identified as an unimodular lattice in $\R^{d+1}$.

\begin{lemma}[{\cite[Lem.~4.4]{Shi20}}] \label{lem:shi 1}
 Let $L>0$. Then there exists $t_0>0$, $0<\e_0<1, \vartheta>0$, $b>0$, and $0<\aaa<1$ such that for all $y \in \Y$, 
    $$
    \int_{\mathrm{M}_{m \times n}([-1,1])} \alpha_{\e_0}^\vartheta(a_{t_0} u(L\theta) y) \, d\theta \leq \aaa \alpha_{\e_0}^\vartheta(y) + b.
    $$
    %\colorbox{red}{Change symbol a to something else}
\end{lemma}

\begin{lemma}[{\cite[Lem.~4.8]{Shi20}}]  \label{lem:shi 2}
Let  $L>0$, $t>0$, $k\in \N$, and $y \in \Y$.  Then for any non-negative function $\psi$ on $\Y$,
    \begin{align*}
        \int_{\mathrm{M}_{m \times n}([-1,1])} \psi(a_{(k+1)t} u(L \theta) y)\, d\theta \leq \int_{\mathrm{M}_{m \times n}([-1,1])} \int_{\mathrm{M}_{m \times n}([-1,1])} \psi(a_{t} u(L\eta) a_{kt} u(L \theta) y)\, d\theta \, d\eta.
    \end{align*}
\end{lemma}

\begin{corollary}
\label{cor: shi}
     Let $L>1$. Then there exists $C_0,C_1>0$, $0<\e_0<1$, and $\vartheta>0$ such that for all $y \in \Y$ and all $t>0$, 
     \begin{align}
         \label{eq:cor: shi}
          \int_{\mathrm{M}_{m \times n}([-L,L])} \alpha_{\e_0}^\vartheta(a_{t} u(\theta) y) \, d\theta \leq C_0 \alpha_{\e_0}^\vartheta(y) + C_1.
     \end{align}
\end{corollary}
\begin{proof}
   Fix $L>1$. Let $t_0>0$, $\e_0>0$, $\vartheta>0$, $b>0$, and $0<c<1$ be as in Lemma~\ref{lem:shi 1}. Also let $C_0'>0$ be a constant such that for all $0 \leq s \leq t_0$, 
   \begin{equation}\label{equ:log_lip}
      \alpha_{\e_0}^\vartheta(a_s y) \leq C_0' \alpha_{\e_0}^\vartheta(y). 
   \end{equation}
The existence of such a constant follows from the properties of $\alpha_\e$ proved in \cite[Lem.~4.1]{Shi20}.

\noindent Fix $t>0$, and write $t= k t_0+ s$, where $k \in \N$ and $0 \leq s<t_0$. Then 
   \begin{align*}
        \int_{\mathrm{M}_{m \times n}([-L,L])} \alpha_{\e_0}^\vartheta(a_{t} u(\theta) y) \, d\theta & =   L^{mn}\int_{\mathrm{M}_{m \times n}([-1,1])} \alpha_{\e_0}^\vartheta(a_{t} u(L\theta) y) \, d\theta  \\
       &\leq L^{mn} C_0' \int_{\mathrm{M}_{m \times n}([-1,1])} \alpha_{\e_0}^\vartheta(a_{kt_0} u(L\theta) y) \, d\theta,  
   \end{align*}
   where we used \eqref{equ:log_lip}.

   \noindent Next, applying Lemma \ref{lem:shi 2} and Lemma \ref{lem:shi 1}, we obtain
 \begin{align*}
   \int_{\mathrm{M}_{m \times n}([-1,1])} \alpha_{\e_0}^\vartheta(a_{kt_0} u(L\theta) y) \, d\theta &\leq   \int_{M_{m \times n}([-1,1])} \left( \int_{\mathrm{M}_{m \times n}([-1,1])} \alpha_{\e_0}^\vartheta(a_{t} u(L\eta) a_{(k-1)t} u(L \theta) y)\, d\eta \right)\, d\theta \\
       &\leq  \aaa \int_{\mathrm{M}_{m \times n}([-1,1])} \alpha_{\e_0}^\vartheta(a_{(k-1)t} u(L\theta) y)\, d\theta    + 2^{mn}b.
   \end{align*}
   Iterating this inequality yields
   \begin{align*}
   \int_{\mathrm{M}_{m \times n}([-1,1])} \alpha_{\e_0}^\vartheta(a_{kt_0} u(L\theta) y) \, d\theta \leq  \aaa^k \int_{\mathrm{M}_{m \times n}([-1,1])} \alpha_{\e_0}^\vartheta(y)\, d\theta    + 2^{mn} b(1+ \aaa+ \cdots+ \aaa^{k-1}).
   \end{align*}
   The corollary now follows by taking $C_0= (2L)^{mn}C_0'$ and $C_1= (2L)^{mn}C_0' b/(1-\aaa)$.
\end{proof}

\subsection{Proof of Proposition~\ref{main prop}}
Throughout this subsection, we fix $0<A<1<B$ and $L>1$. Also fix constants $C_0, C_1>0$, $\e_0>0$, and $\vartheta>0$ as in Corollary~\ref{cor: shi}.

\noindent For $y= [g,gv]\tilde{\Gamma} \in \tmX$, define 
$$
\tilde{y} := \begin{pmatrix}
 (\e_0^{-1} B)^{1/d} I_d   \\ & \e_0 B^{-1}
\end{pmatrix} \begin{pmatrix}
    g & gv \\ & 1
\end{pmatrix} \SL_{d+1}(\Z) \in \Y.
$$

\begin{lemma}
\label{lem:shi imp 1}
    Suppose $y= [g,gv]\tilde{\Gamma} \in \tmX$ with $v \notin \Liou(\R^d)$. Then there exists $\beta= \beta(y)>0$ such that
    $$
    \alpha_{\e_0}^\vartheta(\tilde{y}) \ll B^{\beta},
    $$
    where the implied constant is independent of $B$.
\end{lemma}
\begin{proof} 
   Fix $y= [g,gv]\tilde{\Gamma} \in \tmX$ with $v \notin \Liou(\R^d)$. Throughout this proof, we fix a coset representative $[g,gv]\tilde{\Gamma}$ of $y$, chosen so that $\|g\|$ is minimal. 
   
   \noindent Since $v$ is not Liouville, there exists $E>0$ such that the inequalities 
   $$
   \|q v +p\| \leq \frac{1}{T^E}, \quad 1 \leq q \leq T,
   $$
   have no integral solution $(p,q) \in \Z^d \times \N$.  
   
   Using~\cite[Lem.~4.1]{Shi20}, fix $C>0$ such that for all $z \in \Y$,  
   $$
   C^{-1} \alpha_{\e_0}^\vartheta(z) \leq \alpha_{\e_0}^\vartheta(gz) \leq C\alpha_{\e_0}^\vartheta(z).
   $$
   
\noindent Thus,  
$$
\alpha_{\e_0}^\vartheta(\tilde{y}) \leq C \alpha_{\e_0}^\vartheta(\Lambda_v),
$$
 where
    $$
    \Lambda_v := \begin{pmatrix}
 (\e_0^{-1} B)^{1/d} I_d   \\ & \e_0 B^{-1}
\end{pmatrix} \begin{pmatrix}
    I_d & v \\ & 1
\end{pmatrix} \SL_{d+1}(\Z) \in \Y.
    $$
Hence it is enough to show that $\alpha_{\e_0}^\vartheta(\Lambda_v) \ll B^{\beta}$,  
   
\noindent To prove this, we claim that for every nonzero vector $w \in \Lambda_v$  
    \begin{align}
        \label{eq:temp:3.1}
        \|w\|  \gg B^{-E} \quad \text{ and } \quad  \varphi_{\e_0}(w) \leq B^{\frac{E}{\delta_{\eta_1}}},
    \end{align}
    where we identify $\Lambda_v$ with corresponding unimodular lattice in $\R^{d+1}$.
    
    Indeed, any such $w$ can be written as 
    \[w=\left((\e_0^{-1} B)^{1/d} (p+qv),\e_0B^{-1} q\right),\] 
     for some $(p,q) \in \Z^d \times \Z$, with
     \[\pi_{\eta_1}(w)= (\e_0^{-1} B)^{1/d} (p+qv) \quad \text{and} \quad \pi_0(w)= \e_0B^{-1} q.\] Now, consider the following cases:
\begin{enumerate}
    \item[1.] $ {\|\pi_0(w)\| =0}$. In this case, we have $w= \left((\e_0^{-1} B)^{1/d} p,0\right)$ for some $p \in \Z^d \setminus \{0\}$. Thus $\|w\| = \|(\e_0^{-1} B)^{1/d} p\| \gg 1> B^{-E}$ and $\varphi_{\e_0}(w)\ll \|\pi_{\eta_1}(w)\|^{-1/\delta_{\eta_1}} \ll 1$.  
    \item[2.] $ {\|\pi_0(w)\| >\e_0}$. Then by definition \eqref{eq:def phi}, $\varphi_{\e_0}(w)=0$ and $\|w\| \geq \|\pi_0(w)\| >  \e_0 \gg B^{-E}$.
    \item[3.] $ {\|\pi_0(w)\| \in (0,\e_0) }$. Then $1 \leq |q| \leq B$. By the choice of $E$, $\|p+ q v\| \geq B^{-E}$. Hence $\|w\| \geq \|\pi_{\eta_1}(w)\| \geq (\e_0^{-1} B)^{1/d} B^{-E} \gg B^{1/d-E}\geq B^{-E}$ and $\varphi_{\e_0}(w) \ll \|\pi_{\eta_1}(w)\|^{-1/\delta_{\eta_1}} \leq B^{\frac{E}{\delta_{\eta_1}}}$.  
\end{enumerate}
Note that in all the above three cases, the implied constants can be chosen independent of $B$.

\noindent Now using~\eqref{eq:temp:3.1}, we get the following:
    \begin{enumerate}
        \item[1.] If $w \in \R^{d+1}$ is a non-zero $\Lambda_v$-monomial, then 
        $$
        \varphi_{\e_0}(w) \ll B^{\frac{E}{\delta_{\eta_1}}}.
        $$
        \item[2.] If $w \in \bigwedge^i \R^{d+1}$ is a non-zero $\Lambda_v$-monomial for $i>1$, then $w= w_1 \wedge \cdots \wedge w_i$ for some $w_1, \ldots, w_i \in \Lambda_v$. Let $\Lambda'_w$ be the sublattice of $\Lambda_v$ generated by $w_1, \ldots, w_i$. By Minkowski's second theorem, there exists $w_1', \ldots, w_i' \in \Lambda_w'$ such that 
        \begin{align}
            \label{eq:temp:3.2}
            \text{vol}(\Span(\Lambda_w')/\Lambda_w') \gg \|w_1'\|\cdots \|w_i'\| \gg B^{-iE},
        \end{align}
        where the last inequality follows from~\eqref{eq:temp:3.1}.
        
        Note that since $\|w\|$ equals $ \text{vol}(\Span(\Lambda_w')/\Lambda_w')$ upto a bounded constant, we get that
        \begin{align}
            \label{eq:temp:3.3}
            \|w\| \gg B^{-iE}.
        \end{align}
        Furthermore, for $i>1$, we have $\pi_0(w)=0$, and hence
        \begin{align}
            \varphi_{\e_0}(w) \ll \max \left\{ B^{iE/\delta_{\eta_i}}, B^{iE/\delta_{\eta_{i-1}}} \right\}. 
        \end{align}
    \end{enumerate}

    \noindent From the above two cases, it is clear that there exists $\beta>0$, depending only on $y$ (and independent of $B$), such that
    $$
    \alpha_{\e_0}^\vartheta(\Lambda_v)= \max \varphi_{\e_0}^\vartheta(w) \ll B^\beta.
    $$
\end{proof}

\begin{proof}[Proof of Proposition~\ref{main prop}]
    Let $\theta \in S(L,A,B)$. Then the lattice 
    $$
    \begin{pmatrix}
        a_t u(\theta) & \\& 1
    \end{pmatrix} \begin{pmatrix}
        g & gv \\ &1
    \end{pmatrix}\Z^{d+1}
    $$
    contains a vector in the region $\cB(A,B)$, say $w$. 
    
    Define
    $$
    v:= \begin{pmatrix}
 (\e_0^{-1} B)^{1/d} I_d   \\ & \e_0 B^{-1}
\end{pmatrix}w 
    $$
   Then $v$ is a monomial in $a_tu(\theta)\tilde{y}$, and satisfies \[\pi_0(v) \leq \e_0 \text{ and } \pi_{\eta_1}(v) \leq \e_0^{-1} B^{1/d}A.\] 
   This implies that
    $$
\alpha_{\e_0}^\vartheta(a_tu(\theta)\tilde{y}) \geq C_2 (A^dB)^{-\kappa},
    $$
    for some $\kappa, C_2>0$, where both $\kappa,$ independent of $y$ and $B$.
    
Hence
    \begin{align}
    \label{eq:temp:2.1}
      |S(L,A,B)| \leq   \frac{(A^dB)^{\kappa}}{C_2 } \int_{\mathrm{M}_{m \times n}([-L,L])} \alpha_\e^\vartheta(a_t u(\theta) \tilde{y}) \, d\theta.
    \end{align}
    By Corollary~\ref{cor: shi},  
    \begin{align}
    \label{eq:temp:2.2}
         \int_{\mathrm{M}_{m \times n}([-L,L])} \alpha_\e^\vartheta(a_t u(\theta) \tilde{y}) \, d\theta \leq C_0 \alpha_\e^\vartheta( \tilde{y}) + C_1,
    \end{align}
    and by Lemma~\ref{lem:shi imp 1} 
    \begin{align}
    \label{eq:temp:2.3}
        \alpha_\e^\vartheta( \tilde{y}) \ll B^\beta.
    \end{align}
Combining \eqref{eq:temp:2.1}, \eqref{eq:temp:2.2}, and \eqref{eq:temp:2.3}, we obtain the desired bound, completing the proof.
\end{proof}

\section{Generic sets}\label{subsec:gen}

Let $v \in \T^d$ and $T>0$. Following Kim \cite{Kim2024}, we define
\[
\zeta(v, T) := \min \left\{ N \in \mathbb{N} : \min_{1 \le q \le N} \| q v \|_{\mathbb{Z}} \le \frac{N^2}{T} \right\},
\]
where $\|\cdot\|_{\mathbb{Z}}$ denotes the supremum distance to $0 \in \mathbb{T}^d$.

For $g := (g_{ij}) \in \SL_d(\mathbb{R})$, set
\[
\|g\| := \max_{1 \le i,j \le d} \big( |g_{ij}|, |(g^{-1})_{ij}| \big).
\]

Given $\kappa > 0$, define
\[
\ccG(\kappa,T,\varepsilon) := \left\{ y \in \tmX : \exists \text{ a representative } y = [h,hb]\tilde{\Gamma} \text{ such that } \|h\| \le \kappa \, \zeta(b,T)^{\delta_1}, \ \zeta(b,T)^{\delta_1} \ge T^\varepsilon \right\},
\]
where $\delta_1 > 0$ is as in Theorem~\ref{thm:kim}.

\medskip

The main goal of this section is to prove the following proposition.
\begin{proposition}
\label{lem:imp stop time}
Let $T \geq 1$ and $\kappa>0$ be given. Suppose $y \in \tmX$ with $y=[g, gv] \tilde{\Gamma}$, where $v \notin \Liou(\R^d)$. Then there exists $\e>0$ and $\delta_2, \delta_3>0$, depending on $y$, such that for all $t \geq 0$, 
$$  \left|\left\{\theta \in \mathrm{M}_{m \times n}([-L,L]): a_tu(\theta) y  \notin \ccG(\kappa,T,\e)  \right\}\right|   \;\ll \;   T^{-\delta_2}+ e^{-\delta_3 t},$$
where the implied constant depends on $y$, $\kappa$, and $L$, but is independent of $t$ and $T$.
\end{proposition}

To prove Proposition \ref{lem:imp stop time}, we introduce, for $M > 0$, the sets
\begin{align*}
K_1^M 
&:= \left\{ y \in \tilde{\X} : \|h\| \geq M \text{ for every coset representative } y = [h, hb]\tilde{\Gamma} \right\}, \\
K_2^M 
&:= \left\{ y \in \tilde{\X} : \exists \text{ a coset representative } y = [h, hb]\tilde{\Gamma} \text{ such that } \|h\| \leq M \text{ and } \kappa\, \zeta(b,T)^{\delta_1} \leq M \right\}.
\end{align*}

We will show that, for $M = T^\varepsilon$, the set $\ccG(\kappa, T, \varepsilon)^c$ is contained in $K_1^M \cup K_2^M$. Consequently, it suffices to estimate the measure of $K_1^M$ and $K_2^M$.

\medskip

\subsection{Estimate on $K_1^M$}
\begin{lemma}
    \label{appendix a: main lemma}
   Let $L>0$ be given. Then there exists a $\delta>0$ such that the following holds for all $t>0$ and $y \in \tmX$,
   \begin{align*}
        \left|\{\theta \in \mathrm{M}_{m \times n}([-L,L]): a_tu(\theta) y  \in K_1^M  \} \right| \ll  M^{-1}+ e^{-\delta t},
    \end{align*}
     where the implied constant depends on $y$ and $L$, but is independent of $t$.
\end{lemma}

The proof of Lemma~\ref{appendix a: main lemma} is an immediate consequence of the effective equidistribution theorem of Kleinbock and Margulis~\cite{KM2}. We begin by introducing some notation.

Let $\X := \SL_d(\R)/\SL_d(\Z)$, and let $\mu_{\X}$ denote the unique $\SL_d(\R)$-invariant probability measure on $\X$. Define the projection map $\pi : \tilde{\X} \to \X$ by
\[
[g, gv]\tilde{\Gamma} \mapsto g\Gamma.
\]

For $M > 0$, define the cusp neighborhood
\begin{equation}\label{equ:cusp_M}
\mathcal{U}_M
:=
\left\{
x \in \X :
\|g\| \ge M \text{ for every } g \in \SL_d(\R) \text{ satisfying } x = g\Gamma
\right\}.
\end{equation}

\begin{lemma}\label{lem:cusp_M}
There exists a constant $C>0$ such that for all $M>0$, 
\[
\mu_{\X}(\mathcal{U}_M) \leq C M^{-1}.
\]
\end{lemma}

\begin{proof}
Let $\mathcal{F}$ denote the fundamental domain for $\X$ defined as in \cite[\S2.7 and Appendix A]{Kim2024}. By definition,
\[
\mathcal{U}_M \subset \{ g\Gamma : g \in \mathcal{F},\ \|g\| \ge M \}.
\]
Therefore,
\[
\mu_{\X}(\mathcal{U}_M)
\le
\mu_{\X}(\{ g\Gamma : g \in \mathcal{F},\ \|g\| \ge M \})
\ll M^{-1},
\]
where the last inequality follows from \cite[Equation (2.16)]{Kim2024}. This completes the proof.
\end{proof}

\begin{theorem}[{\cite[Prop.~2.4.8]{KM2}}]\label{thm:KM}  
 Let $L>1.$ Then there exist constants  $\delta>0$ and $k \in \mathbb N$ such that 
for every $x\in \mathcal{X}$, every $f\in C_c^\infty(\mathcal{X})$, and every $t>0$, one has
\[
\frac{1}{ (2L)^{mn}}\int_{\mathrm{M}_{m \times n}([-L,L])
} f(a_t u(\theta) x)\,d\theta
\;=\;
 \int_{\mathcal{X}} f \,d\mu_{\mathcal{X}}
+\,
O_{d,x}\left(
   e^{-\delta t}\,\|f\|_{C^k}
\right).
\]
\end{theorem}
\medskip

\begin{proof}[Proof of Lemma~\ref{appendix a: main lemma}]
Fix $y \in \tmX$ and $L>0$, and set $x= \pi(y)$. Then  
\begin{align}
    \label{eq:temp:5.1}
    \left|\{\theta \in \mathrm{M}_{m \times n}([-L,L]): a_tu(\theta) y  \in K_1^M  \} \right|= \left|\{\theta \in \mathrm{M}_{m \times n}([-L,L]): a_tu(\theta) x  \in \mathcal{U}_M  \} \right|,
\end{align}
which we now estimate.

\noindent To this end, fix $\varphi \in C_c^{\infty}(G)$ supported on a sufficiently small neighborhood of the identity such that
\[
\int_G \varphi(g)\,dm_G(g)=1
\quad\text{and}\quad
\|\varphi\|_{C^k}<\infty \ \text{for all } k\in\mathbb N,
\]
where $m_G$ denotes the Haar measure on $G$, normalized so that $m_G(\mathcal{F})=1.$

\noindent Let $\beta>1$ be such that for every $g\in \supp(\varphi)$ 
\begin{align} 
\mathcal{U}_{\beta M}
\subseteq
g^{-1}\cdot \mathcal{U}_M
\subseteq
\mathcal{U}_{\beta^{-1}M}.
\end{align}

\noindent Define
\[
f_M
:=
\varphi * \left(1-\ind_{\mathcal{U}_{\beta^{-1} M}}\right) = 1- \varphi * \ind_{\mathcal{U}_{\beta^{-1} M}}.
\]

Note that if $\theta \in \mathrm{M}_{m \times n}([-L,L])$ is such that $a_tu(\theta)x \in \mathcal{U}_M$. Then for all $g \in \supp(\varphi)$, 
\[g^{-1}a_tu(\theta)x \in \mathcal{U}_{\beta^{-1} M},\]  which implies 
$$
f_M(a_tu(\theta)x)= \int_G \varphi(g) \left( 1-\ind_{\mathcal{U}_{\beta^{-1}M}}(g^{-1}a_tu(\theta)x)  \right)\, dm_G(g) =0.
$$
Since $f_M \leq 1$, it follows that
\begin{align}
    \label{eq:temp:5.2}
     \left|\{\theta \in \mathrm{M}_{m \times n}([-L,L]): a_tu(\theta) x  \in \mathcal{U}_M  \} \right| \leq \int_{\mathrm{M}_{m \times n}([-L,L])} (1-f_M)(a_t u(\theta) x) \, d\theta.
\end{align}

Note that $f_M$ is a smooth function and $\|f_M\|_{C^k}$ is bounded by a constant independent of $M$ for any $k \in \mathbb N$. Moreover,
   \begin{align}
       \mu_{\X}(f_M) &= \int_\X \int_G \varphi(g) \left( 1- \ind_{ \mathcal{U}_{\beta^{-1} M}}(g^{-1} x) \right) \, dm_G(g) d\mu_\X(x) \nonumber \\
       &=  \int_G \varphi(g) \left( \int_\X \left( 1- \ind_{ \mathcal{U}_{\beta^{-1} M}}(g^{-1} x) \right) \, d\mu_\X(x) \right) dm_G(g) , \quad \text{using Fubini's theorem} \nonumber \\
       &= \int_G \varphi(g) \left( 1- \mu_{\X}( \mathcal{U}_{\beta^{-1} M}) \right) dm_G(g), \quad \text{using $\SL_d(\mathbb R)$-invariance of $\mu_{\X}$} \nonumber \\
       &= 1- \mu_{\X}( \mathcal{U}_{\beta^{-1} M}). \label{eq: a 3}
   \end{align}

   By Theorem~\ref{thm:KM}, there exists $\delta>0$ such that
   \begin{align}
       \frac{1}{(2L)^{mn}} \int_{ \mathrm{M}_{m \times n}([-L,L]) } f_{M}( a_{ t} u(\theta) x) \, d\theta &= \mu_{\X}(f_{M}) + O\left(
  e^{-\delta t }\,\|f_M\|_{C^k}
\right) \nonumber \\
&=  1- \mu_{\X}( \mathcal{U}_{\beta^{-1} M}) + O(e^{-\delta_0 t} ), \label{eq: a 5}
\end{align}
where the implied constant is independent of $M$.

\noindent The lemma now follows from \eqref{eq:temp:5.1},~\eqref{eq:temp:5.2}, \eqref{eq: a 5}, and Lemma~\ref{lem:cusp_M}.
\end{proof}

\subsection{Estimate on $K_2^M$}
\begin{lemma}
    \label{appendix a: main lemma 2}
   Let $L>0$ be given. Then there exists constants $c_1,c_2>0$ (depending only on $y$ and $L$) such that the following holds for all $t>0$ and $y \in \tmX$,
   \begin{align*}
        \left|\{\theta \in \mathrm{M}_{m \times n}([-L,L]): a_tu(\theta) y  \in K_2^M  \} \right| \ll  M^{c_1}T^{-c_2},
    \end{align*}
     where the implied constant depends on $y, \kappa,$ and $ L$, but is independent of $t$.
\end{lemma}
\begin{proof}
    Suppose that $a_tu(\theta)y \in K_2^M$. Then there exists a coset representative $[h,hb]\tilde{\Gamma}$ of $a_tu(\theta)y$ such that $\|h\| \leq M$ and $\kappa \,\zeta(b,T)^{\delta_1} \leq M$. The second condition implies that
     $$
     \min\left\{ N \in \N: \min_{1 \leq q \leq N} \|qb\|_{\mathbb Z} \leq \frac{N^2}{T} \right\} \leq (\kappa^{-1}M)^{1/\delta_1}.
     $$
     Consequently, there exists a pair $(p,q) \in \Z^d \times \N$ such that
     \begin{align*}
         \|qb+p\| \leq  \frac{(\kappa^{-1}M)^{2/\delta_1}}{T}, \quad \text{and } 1 \leq q \leq (\kappa^{-1}M)^{1/\delta_1}.
     \end{align*}
     Together with the bound $\|h\|\leq M$, this implies that the lattice
     \begin{align*}
         \begin{pmatrix}
             h & hb \\ & 1
         \end{pmatrix} \Z^{d+1} = \begin{pmatrix}
             a_t u(\theta) \\ &1
         \end{pmatrix} \begin{pmatrix}
             g & gv \\& 1
         \end{pmatrix} \Z^{d+1}
     \end{align*}
     has a non-zero point in the region 
     \begin{align}
         \label{eq:temp:1.3}
          \left\{ (x,y) \in \R^d \times \R: \|x\| \leq \frac{(\kappa^{-1}M)^{2/\delta_1 +1}}{T} \text{ and } |y| \leq (\kappa^{-1}M)^{1/\delta_1} \right\}.
     \end{align}
     Therefore,
     \begin{align}
         \label{eq:temp:1.2}
         & \left\{\theta \in \mathrm{M}_{m \times n}([-L,L]): a_tu(\theta) y  \in K_2^M  \right\} \subset S(L,{(\kappa^{-1}M)^{2/\delta_1+1}}/{T}, (\kappa^{-1}M)^{1/\delta_1}),
     \end{align}
     where $S(\cdot)$ is defined as in Proposition~\ref{main prop}. Now applying the measure estimate for $S(\cdot)$ from Proposition~\ref{main prop}, we obtain
     \begin{align}
     \label{eq:temp:1.5}
           \left|\{\theta \in \mathrm{M}_{m \times n}([-L,L]): a_tu(\theta) y  \in K_2^M  \} \right| \leq  \beta_0 \left(\frac{(\kappa^{-1}M)^{2/\delta_1}}{T} \right)^{\beta_1} (\kappa^{-1}M)^{\beta_2/\delta_1}, 
     \end{align}
     for some constants $\beta_0, \beta_1, \beta_2$ depending only on $y$ and $L$. This proves the lemma.
\end{proof}

\subsection{Proof of Proposition~\ref{lem:imp stop time}}
\begin{proof}
 First, note that
 \begin{align}
        \label{eq:temp:1.0}
         \ccG(\kappa, T, \e)^c \subset K_1^M \cup K_2^M,
    \end{align}
provided we choose $M= T^\e$.

 Indeed, suppose $y \notin \ccG(\kappa, T,\e)$. Then either every coset representative $[h,hb]\tilde{\Gamma}$ of $y$ satisfy $\|h\| \geq M$, or there exists at least one coset representative $[h,hb]\tilde{\Gamma}$ of $y$  with $ \|h\| \leq M$. In the latter case, the definition forces  $ \kappa \,\zeta(b,T)^{\delta_1}<M$. Thus, in the first case, $y \in K_1^M$, while in the second case  $y \in K_2^M$. 
 
Therefore, the proposition follows directly from Lemma~\ref{appendix a: main lemma} and Lemma~\ref{appendix a: main lemma 2} by choosing   ${\e}$ sufficiently small in the relation $M= T^\e$. Note that the choice of $\e$ is independent of $T$ and $t$, although it depends on $y$ and $L$.   
\end{proof}

\section{Proof of Theorem~\ref{thm: main thm equidistribution}}\label{sec:eff_equ}

The key ingredient in the proof of Theorem~\ref{thm: main thm equidistribution} is the following result of Kim~\cite{Kim2024}. 

\begin{theorem}[{\cite[Thm.~5.2]{Kim2024}}]\label{thm:kim}
 Let $V$ be a fixed neighbourhood of $0$ in $\mathrm{M}_{m \times n}(\R)$ with smooth boundary and compact closure. Then there exist $\delta_1>0$ (depending only on $m$ and $n$) and $k \geq 1$ such that the following holds.

\noindent For every $f \in C_c^{\infty}(\tilde{\mathcal{X}})$, every $t \geq 0$, and every
\[
y=[g,gv] \tilde{\Gamma} \in \tilde{\mathcal{X}} \qquad \text{with} \qquad \|g\| \leq \zeta(v,e^{wt})^{\delta_1},
\]
we have
\begin{equation*}
\frac{1}{|V|}
\int_{V}
f(a_t u(\theta)y)\,d\theta
=
\mu_{\tilde{\mathcal{X}}}(f)
+
O_{d,V}\!\left(\mathcal{S}_k(f)\,\zeta(v,e^{wt})^{-\delta_1}\right),
\end{equation*}
where $w$ is defined as in \eqref{eq: def w}.
%The implied constant depends only on $m$, $n$, and $V$.
\end{theorem}

\subsection{Double-Equidistribution}
\begin{proposition}
    \label{prop:base step of Induction}
  Let $y= [g,gv]\tilde{\Gamma} \in \tmX$ with $v \notin \Liou(\R^d)$, and let $V$ be a compact subset of $\Mat$. Then there exists $\delta_0= \delta_0(y)>0$ (depending only on $m$ and $n$) and $k \geq 1$ such that the following holds.  

For every $f_0 \in C_c^\infty( \Mat)$ with $\supp(f_0) \subset V$, $f_1 \in C_c^{\infty}(\tilde{\mathcal{X}})$, $r,s > 0$, we have
\begin{align*}
\int_{\Mat} f_0(\theta)
 f_1(a_{r+s} u(\theta)y) &\,d\theta
=
  \int_{\Mat}f_0 \, d\theta \, \int_{\tmX} f_1 \,d\mu_{\tmX}   \\
&+
O_{y,V}\!\left( \|f_0\|_{w,s} \|f_1\|_{C^0}+ e^{-\delta_0 \min\{s,r\}}\|f_0\|_{C^0} \mathcal{S}_k(f_1)\right),
\end{align*}
 where $\|\cdot\|_{w,t}$ is defined as in \eqref{eq:def weighted norm}.
\end{proposition}
\begin{proof}
 Fix a compact subset $V$ of $\Mat$. Without loss of generality, assume $V= \mathrm{M}_{m \times n}([-L,L])$ for some $L \geq 1$. Let $\delta_1>0$ be as in Theorem~\ref{thm:kim} for $V= \mathrm{M}_{m \times n}([-L,L])$. Also let $\e>0$ and $\delta_2, \delta_3>0$ be as in Proposition~\ref{lem:imp stop time}. Let $\kappa>0$ be a constant such that for all $g \in \SL_d(\R)$, $t>0$ and $\theta \in V$, we have \[\| u(\theta) g\| \geq \kappa \|g\|.\]
For $s>0$, define
    \[J:= \mathrm{M}_{m \times n}([-1/2,1/2]), \,\, J_s:= b_s^{-1} J c_s^{-1}.\]
    Then there exists a finite set $I_s \subset \Mat$ and a remainder set $V_s \subset V$ such that  
    \[V = V_s \sqcup \bigsqcup_{\eta \in I_s} \left(  \eta+ J_s \right),\]
where the Lebesgue measure of the remainder set $V_s$ is less than $O(e^{-4ws})$ with implied constant depends only on $V$.

\noindent Fix $f_0 \in C_c^\infty(\mathrm{M}_{m \times n}(\R))$ with $\supp(f_0) \subset V$, and $f_1 \in C_c^{\infty}(\tilde{\mathcal{X}})$. Then
\begin{align}
  &  \int_{\Mat} f_0(\theta)
 f_1(a_{r+s} u(\theta)y)\,d\theta \nonumber \\ & = \int_{V} f_0(\theta)
 f_1(a_{r+s} u(\theta)y)\,d\theta \nonumber \\
 &= \sum_{\eta \in I_s} \int_{\eta +J_s} f_0(\theta)
 f_1(a_{r+s} u( \theta)  y)\,d\theta+ O(|V_s| \|f_0\|_{C^0} \|f_1\|_{C^0} ) \nonumber\\
 &= |J_s| \sum_{\eta \in I_s} \int_{J} f_0(\eta+ b_s^{-1}\theta c_s^{-1})
 f_1(a_{r} u( \theta)a_s u(\eta)  y)\,d\theta + O(e^{-4ws}\|f_0\|_{C^0} \|f_1\|_{C^0} ), \text{ by change of variable} \nonumber\\
 &= |J_s| \sum_{\eta \in I_s} \int_{J} \left(\int_J f_0(\eta+ b_s^{-1}\omega c_s^{-1})\, d\omega +  O(\|f_0\|_{w,s}) \right)  
 f_1(a_{r} u( \theta)a_s u(\eta)  y)\,d\theta + O(e^{-4ws}\|f_0\|_{C^0} \|f_1\|_{C^0} )  \nonumber \\
 &= |J_s| \sum_{\eta \in I_s} \left(\int_J f_0(\eta+ b_s^{-1}\omega c_s^{-1})\, d\omega \right) \left( \int_J  f_1(a_{r} u( \theta)a_s u(\eta)  y)\,d\theta \right) + O( \|f_0\|_{w,s}\|f_1\|_{C^0}+ e^{-4ws} \|f_0\|_{C^0} \|f_1\|_{C^0} ). \label{eq:temp:4.-1}
\end{align}

  We now partition the index set $I_s$ into two disjoint subsets:
\[I_s :=I_s' \sqcup I_s'',\]
where \[I_s'=\{\eta \in I_s: a_s u(\eta)y \in \ccG(1,e^{wr}, \e)\} \quad \text{and } I_s''=I_s \setminus I_s'.\]
Recall that $ \ccG( \cdot)$ is defined in \S \ref{subsec:gen}.

  For $\eta \in I_s'$, Theorem~\ref{thm:kim} gives 
\begin{align}
    \label{eq:temp:4.0}
     \int_J  f_1(a_{r} u( \theta)a_s u(\eta)  y)\,d\theta = \int_{\tmX} f_1 \, d\mu_{\tmX} + O(\mathcal{S}_k(f_1)e^{-\e w r}).
\end{align}
We now estimate the contribution from the complement $I_s''$. For $\eta \in I_s''$, we have $ a_su(\eta)y \notin \ccG(1,e^{wr},\e)$, which implies that
\[a_su(\theta)y \notin \ccG(\kappa, e^{wr},\e) \,\,\text{  for all } \theta \in \eta+ J_s.\]
Hence, applying Proposition~\ref{lem:imp stop time}, we obtain 
\begin{eqnarray*}
    |J_s| \cdot  \#I_s'' &\leq &   \left| \{ \eta \in V:  a_su(\eta)y \notin \ccG(\kappa,e^{wr},\e)\} \right| +|V_s|\\ &\ll & e^{-\delta_2wr} + e^{-\delta_3s} + e^{-ws}.
\end{eqnarray*}
 
\noindent Therefore, we get that
\begin{align}
     &|J_s| \sum_{\eta \in I_s''} \left(\int_J f_0(\eta+ b_s^{-1}\omega c_s^{-1})\, d\omega \right) \left( \int_J  f_1(a_{r} u( \theta)a_s u(\eta)  y)\,d\theta \right)  \nonumber\\
     &\ll\|f_0\|_{C^0}\|f_1\|_{C^0} (e^{-\delta_2wr} + e^{-\delta_3s} + e^{-ws}). %\vol(J_s) \sum_{\eta \in I_s''} \left(\int_J f_0(\eta+ b_s^{-1}\omega c_s^{-1})\, d\omega \right) \left( \int_\tmX  f_1\,d\mu_{\tmX} \right) 
     \label{eq:temp:4.1}
\end{align}

Take $\delta_0= \min\{\e w,  \delta_2 w, \delta_3, w\}$. Then combining~\eqref{eq:temp:4.-1},~\eqref{eq:temp:4.0}, and~\eqref{eq:temp:4.1}, we get
\begin{align}
  &  \int_{\Mat} f_0(\theta)
 f_1(a_t u(\theta)y)\,d\theta \nonumber\\
 &=  |J_s| \sum_{\eta \in I_s} \left(\int_{J} f_0(\eta+ b_s^{-1}\omega c_s^{-1})\, d\omega \right) \left( \int_\tmX  f_1\,d\mu_{\tmX} \right) \nonumber \\ &+ O \left(  \|f_0\|_{w,s}\|f_1\|_{C^0}+ e^{-\delta_0 \min\{s,r\}} \|f_0\|_{C^0} \mathcal{S}_k(f_1)+ e^{-\delta_0 \min\{s,r\}} \|f_0\|_{C^0} \|f_1\|_{C^0}\right) \nonumber\\
 &=\sum_{\eta \in I_s} \left(\int_{\eta+J_s} f_0( \omega)\, d\omega \right) \left( \int_\tmX  f_1\,d\mu_{\tmX} \right)+O\left(\|f_0\|_{w,s}\|f_1\|_{C^0}+ e^{-\delta_0 \min\{s,r\}} \|f_0\|_{C^0} \mathcal{S}_k(f_1)\right) \nonumber \\&=  \left(\int_{\Mat} f_0(\omega)\, d\omega \right) \left( \int_\tmX  f_1\,d\mu_{\tmX} \right) + O\left( \|f_0\|_{w,s}\|f_1\|_{C^0}+ e^{-\delta_0 \min\{s,r\}} \|f_0\|_{C^0} \mathcal{S}_k(f_1) \right), \nonumber
\end{align}
  the second last equality follows using the estimate $|V_s| \ll e^{-\delta_0 s}$. Hence, the proposition follows.  
\end{proof}

\subsection{Proof of Theorem~\ref{thm: main thm equidistribution}}

\begin{proof}[Proof of Theorem~\ref{thm: main thm equidistribution}]
    Fix $y= [g,gv]\tilde{\Gamma} \in \tmX$ with $v \notin \Liou(\R^d)$, and let $V$ be a neighbourhood of zero in $\Mat$. Choose $L>1$  such that $V \subset \mathrm{M}_{m \times n}([-L,L])$. Let $\delta_0$ and $k$ be as in Prop.~\ref{prop:base step of Induction}. Without loss of generality, we may assume that $\delta_0 < 4w$.

  Fix $F_0 \in C_c^\infty(\Mat)$ and $F_1, \ldots, F_\ell \in C_c^\infty(\tmX)$ with $\supp(F_0) \subset V$. Fix $t_1< t_2 < \ldots < t_\ell$ and let $t= D(t_1, \ldots, t_\ell)$ be as in Theorem \ref{thm: main thm equidistribution}.

\noindent  Apply Proposition~\ref{prop:base step of Induction} with 
    $$
    f_0(\theta)= F_0(\theta) F_1(a_{t_1}u(\theta) y) \ldots F_{\ell-1}(a_{t_{\ell-1}}u(\theta)y),\,\, f_1= F_\ell, \,\,s= \frac{t_\ell+t_{\ell-1}}{2} , \,\,r=\frac{t_\ell-t_{\ell-1}}{2},
    $$
gives  
    \begin{align*}
         &\int_{\Mat}  F_0(\theta) F_1(a_{t_1} u(\theta) y) \cdots F_l(a_{t_\ell}u(\theta)y) \, d\theta \\
         &= \left(\int_{\Mat}  F_0(\theta) F_1 (a_{t_1} u(\theta) y) \cdots F_{\ell-1}( a_{t_{\ell-1} }  u(\theta)y) \, d\theta \right) \int_{\tmX} F_\ell \,d\mu_{\tmX} \\
         &+ O\!\left(\|F_0\|_{C^0}\cdots \|F_{\ell-1}\|_{C^0}  \mathcal{S}_k(F_{\ell-1})\,e^{-\delta_0 t}+\| f_0\|_{w,s}\|F_l\|_{C^0}  \right)
         \end{align*}
Now using Lemma~\ref{lem: C^k norm of restricted function}, we get
$$
\| f_0\|_{w,s} \ll e^{-4w (s-t_{l-1})} \|F_0\|_{C^1}\dots\|F_{\ell-1}\|_{C^1}\leq e^{-\delta_0 t} \|F_0\|_{C^1}\dots\|F_{\ell-1}\|_{C^1}.
$$
Thus, we get that 

         \begin{align*}
            &\int_{\Mat}  F_0(\theta) F_1(a_{t_1} u(\theta) y) \cdots F_\ell(a_{t_\ell}u(\theta)y) \, d\theta \\
            &=  \left(\int_{\Mat}  F_0(\theta) F_1 (a_{t_1} u(\theta) y) \cdots F_{\ell-1}( a_{t_{\ell-1} }  u(\theta)y) \, d\theta \right) \int_{\tmX} F_\ell \,d\mu_{\tmX} + O\!\left(\|F_0\|_{C^1} \mathcal{S}_k(F_1)\cdots  \mathcal{S}_k(F_\ell)\,e^{-\delta_0 t} \right).
         \end{align*}
         The theorem now follows by induction.  
\end{proof}

\section{Notation II} \label{sec:notation 2}
\subsection{Siegel Transform}
 Fix $d:=m+n$ as earlier. For the rest of the paper, we now identify $\tmX$ with the space of all affine unimodular lattices in $\R^d$, via the map
 $$
 [A,b]\tilde{\Gamma} \mapsto A\Z^d +b.
 $$
 
 For a bounded measurable function with compact support $f: \mathbb R^{d} \rightarrow \mathbb R$, the Siegel transform $\hat{f}$ of the function $f$ is a function from $\tmX $ to $ \mathbb R$ defined as
\[
\hat{f}(\Lambda)=\sum_{x \in \Lambda } f(x) \quad \text{for}\,\, \Lambda \in \tmX.
\]
We need the following Siegel Mean Value Theorem and Rogers formula for affine lattices (see \cite[Lemma 4]{JA} and \cite[Formula 3.7]{BMV}).
\begin{proposition}\label{prop:roger}
    For $f \in L^1(\mathbb R^d) \cap L^2(\mathbb R^d),$ we have the following:
    \[\int_{\tmX}\hat{f}(\Lambda) \, d\mu_{\tmX}(\Lambda)=\int_{\mathbb R^d}f(\theta) \, d\theta,\]
    \[\int_{\tmX}\hat{f}(\Lambda)^2 \, d\mu_{\tmX}(\Lambda)=\left(\int_{\mathbb R^d}f(\theta) \, d\theta \right)^2+ \int_{\mathbb R^d}f(\theta)^2 \, d\theta.\]
\end{proposition}

\subsection{Height Function}
We also recall the definition of the height function $\alpha_0$ on $\X:= \SL_{d}(\R)/\SL_d(\Z)$, which is identified as the space of all unimodular lattices in $\R^d$ via the map
$$
A \cdot \SL_d(\Z) \mapsto A\Z^d.
$$
Recall that given a unimodular lattice $\Lambda$ of $\mathbb R^d,$ we say that a subspace $V$ of $\mathbb R^d$ is $\Lambda$-rational if $V \cap \Lambda$ is a lattice in $V.$ Let $d_{\Lambda}(V)$ denotes the volume of $V/(V \cap \Lambda)$ when $V$ is $\Lambda$-rational. We define $\alpha_0 : \X \rightarrow \mathbb R$ by
\begin{equation}\label{equ:alpha_d}
    \alpha_0(\Lambda):=\sup \{d_{\Lambda}(V)^{-1}:\, V \text{ is }\Lambda\text{-rational subspace of }\mathbb R^d \}.
\end{equation}
Further we define ${\alpha}$ on $\tmX$ as $${\alpha}([A,b]\tilde{\Gamma}):= \alpha_0(A\cdot \SL_d(\Z)).$$

The following proposition describes how the Seigel transform grows with respect to the height function.
\begin{proposition}[{\cite[Prop.~3.14]{AG23}}]\label{prop:hat_f_bound_alpha}
    Suppose that $f:\mathbb R^{d} \rightarrow \mathbb R$ be a bounded measurable function with compact support. Then
    \[|\hat{f}(\Lambda)| \ll_{\supp(f)} \|f\|_{C^0} \alpha(\Lambda) \quad \text{for all} \,\, \Lambda \in \tmX.\]
\end{proposition}

\begin{proposition}[{\cite[Prop.~3.17]{AG23}}]\label{prop:alpha_in_L_p}
    The function $\alpha $ is in $L^p(\tmX)$ for $1\leq p < d.$ In particular,
    \[\mu_{\tmX}(\{\alpha \geq L\}) \ll_p L^{-p} \quad \text{for all} \,\,1 \leq p<d.\]
\end{proposition}

\begin{lemma}\label{lem:Y_alpha_bound}
    There exists $\kappa>0$ such that for all $L \geq 1$ and $s \geq \kappa \log L$,
    \[\mu_{\Y_{\xi}}(\{y \in \Y_{\xi}: \alpha(b^s y) \geq L\}) \ll_p L^{-p}\,\, \text{ for all}\,\, 1 \leq p<d.\]
\end{lemma}
\begin{proof}
By definition 
\begin{align*}
    \mu_{\Y_{\xi}}(\{y \in \Y_{\xi}: \alpha(b^s y) \geq L\}) =  \left|\{\theta \in \mathrm{M}_{m \times n}([0,1]): {\alpha}(a_s u(\theta){\Gamma} ) \geq L\} \right| \nonumber.
    \end{align*}
The lemma therefore follows directly from \cite[Prop.~4.5]{BG}.
\end{proof} 
\subsection{Truncated Siegel Transforms}
The Siegel transform of a compactly supported function is, in general, unbounded on $\tmX$. To handle this difficulty, one typically approximates $\hat{f}$ by compactly supported functions on $\tmX$, called \emph{truncated Siegel transforms}, denoted by $\hat{f}^{(L)}$. These are defined with the aid of a smooth cut-off function $\eta_L$, introduced in the following lemma.
\begin{lemma}
   For every $c>1,$ there exists a family $(\eta_L)$ of $C_c^{\infty}(\tmX)$ functions satisfying:
   \[0\leq \eta_L \leq 1,\,\,\,\, \eta_{L}=1\,\,\text{on}\,\, \{\alpha \leq c^{-1}L\},\,\,\,\,\eta_{L}=0\,\,\text{on}\,\, \{\alpha \geq cL\},\,\,\,\, \|\eta_L\|_{C^k} \ll 1 \]
\end{lemma}
Now, given any bounded measurable function $f:\mathbb R^{d} \rightarrow \mathbb R$ with compact support, we define the \emph{truncated Siegel transform} of $f$ as 
\[\hat{f}^{(L)}:=\hat{f}\cdot\eta_L.\]

We recall the basic properties of truncated Siegel transforms, whose proof follows from \cite[Lemma 4.12]{BG} and \cite[Lemma 3.12]{AG23}.
\begin{lemma}\label{lem:pro_f_hat_L}
For $f \in C_c^\infty(\mathbb{R}^{d})$, the truncated Siegel transform $\hat f^{(L)}$ belongs to $C_c^\infty(\tmX)$, and it satisfies
\[\left\|\hat f^{(L)}\right\|_{L^p(\tmX)} \;\leq\; \left\|\hat f\right\|_{L^p(\tmX)} 
   \;\ll_{\mathrm{supp}(f),p}\; \|f\|_{C^0}, 
   \qquad \text{for all } 1 \leq p < d,\]
\[ \left\|\hat f^{(L)}\right\|_{C^0} \;\ll_{\mathrm{supp}(f)}\; L \, \|f\|_{C^0}, \]
\[ \left\|\hat f^{(L)}\right\|_{C^k} \;\ll_{\mathrm{supp}(f)}\; L \, \|f\|_{C^k},\]
\[\mathcal{S}_k \left(\hat f^{(L)}\right) \ll_{\supp (f)} L^{k+1}\|f\|_{C^k},\]
\[ \left\|\hat f - \hat f^{(L)} \right\|_{L^1(\tmX)} \;\ll_{\mathrm{supp}(f),\tau}\; 
   L^{-\tau} \, \|f\|_{C^0}, 
   \qquad \text{for all } \tau < d-1,\]
\[ \left\|\hat f - \hat f^{(L)}\right\|_{L^2(\tmX)} \;\ll_{\mathrm{supp}(f),\tau}\; 
   L^{-(\tau-1)/2} \, \|f\|_{C^0}, 
   \qquad \text{for all } \tau < d-1.
\]
Moreover, the implied constants are uniform when $\mathrm{supp}(f)$ is contained in a fixed compact set.
\end{lemma}

\subsection{Cumulants}
 To prove the central limit theorem, we employ the method of cumulants. We begin by recalling the definition of cumulants.
\begin{definition}
    Let $(\mathcal{Z}, \nu)$ be a probability space. Given bounded measurable functions $\phi_1, \ldots , \phi_r$ on $\mathcal{Z}$, we define their \textit{joint cumulant} as $$\cum_{\nu}^{(r)}(\phi_1, \ldots, \phi_r) := \sum_{\Pp} (-1)^{|\Pp|-1} (|\Pp|-1)! \prod_{I \in \Pp} \int_{\mathcal{Z}} \left( \prod_{i \in I} \phi_i \right) \, d\nu, $$ where the sum is taken over all partitions $\Pp $ of the set $ \{1, \ldots , r\}.$ For a bounded measurable function $\phi$ on $\mathcal{Z}$, we also set $$\cum^{(r)}_\nu(\phi) := \cum^{(r)}_\nu (\phi, \ldots, \phi).$$
\end{definition}
We have the following classical CLT criterion (see \cite{FS}).
\begin{proposition}[{\cite[Prop.~3.4]{BG}}]\label{prop:Method of Cumulants}
 Let $(F_T)$ be a sequence of real-valued bounded measurable functions such that $$\int_{\mathcal{Z}} F_T \, d\nu=0, \quad  \sigma^2:= \lim_{T \rightarrow \infty} \int_{\mathcal Z} F_T^2 \, d\nu < \infty,$$ and $$\lim_{T \rightarrow \infty} \cum^{(r)}_\nu(F_T) = 0, \text{    for all } r \geq 3.$$ Then for every $\eta \in \R$, $$\nu(\{z \in \mathcal Z: F_T(z) < \eta \}) \rightarrow \mathcal{N}_{\sigma}(\eta) \quad\text{   as } T \rightarrow \infty,$$ 
 where \[
\mathcal N_{\sigma}(\eta)
:= \frac{1}{\sqrt{2\pi\sigma}}
\int_{-\infty}^{\eta} e^{-x^2/(2\sigma)}\,dx
\]
denotes the normal distribution with mean zero and variance $\sigma.$ 
\end{proposition}

\section{Proof of Theorem~\ref{thm: main thm}}\label{sec:mainproof}

\subsection{The space $\Y_\xi$}
Throughout this section, fix $\xi \notin \Liou(\R^m)$ and define
\[
\Lambda_{\theta,\xi}
:=
\left\{
\left(
p_1+\sum_{j=1}^{n}\theta_{1j}q_j+\xi_1,\dots,
p_m+\sum_{j=1}^{n}\theta_{mj}q_j+\xi_m,
q_1,\dots,q_n
\right)
: p_i,q_i \in \Z
\right\}.
\]
On identifying $\tilde{\X}$ with the space of affine unimodular lattices in $\R^{m+n}$, we get
\[
\Lambda_{\theta,\xi}
=
\left[ u(\theta), 
\begin{pmatrix}
\xi \\
0
\end{pmatrix}
\right]\tilde{\Gamma}.
\]

Let $\Y_\xi$ denote the collection of affine lattices $\Lambda_{\theta,\xi}$ as $\theta$ ranges over $\Mat$. Geometrically, $\Y_\xi$ is an $mn$-dimensional torus inside $\tilde{\X}$. We denote by $\mu_{\Y_\xi}$ the probability measure on $\Y_\xi$ induced by Lebesgue measure on $\Mat$; equivalently, it is the pushforward of Lebesgue measure on $\mathrm{M}_{m\times n}([0,1])$.

As a consequence of Theorem~\ref{thm: main thm equidistribution}, we obtain the following.

\begin{corollary}\label{cor:Y_mul_mixing}
Fix $\xi \notin \Liou(\R^m)$. Then there exist $\delta = \delta(\xi) > 0$ and $k \in \N$ such that for every $F_0 \in C_c^\infty(\Y_\xi)$, $F_1,\dots,F_\ell \in C_c^\infty(\tilde{\X})$, and $t_1,\dots,t_\ell > 0$, one has
\begin{align*}
\int_{\Y_\xi} F_0(y)\prod_{i=1}^{\ell} F_i(a_{t_i} y)\, d\mu_{\Y_\xi}(y)
&=
\left(\int_{\Y_\xi} F_0 \, d\mu_{\Y_\xi}\right)
\left(\prod_{i=1}^{\ell} \int_{\tilde{\X}} F_i \, d\mu_{\tilde{\X}}\right) \\
&\quad + O_{\xi,\ell}\!\left(
e^{-\delta D(t_1,\dots,t_\ell)} \|F_0\|_{C^1}
\prod_{i=1}^{\ell} \mathcal{S}_k(F_i)
\right),
\end{align*}
where $D(t_1,\dots,t_\ell)$ is as in Theorem~\ref{thm: main thm equidistribution}.
\end{corollary}

\begin{proof}
The argument is analogous to that of \cite[Cor.~2.4]{BG}, with Theorem~\ref{thm: main thm equidistribution} in place of \cite[Thm.~2.2]{BG}, and is therefore skipped.
\end{proof}

\subsection{Reduction to a CLT for Birkhoff sums}

We now reduce the proof of Theorem~\ref{thm: main thm} to a central limit theorem for Birkhoff sums of certain Siegel transforms. To make this precise, define for $T \geq 1$
\[
\Omega_T
:=
\left\{
(\mathbf{x},\mathbf{y}) \in \R^{m+n} :
1 \le \|\mathbf{y}\| \le T,\ 
|x_i| < \vartheta_i \|\mathbf{y}\|^{-w_i},\ i=1,\dots,m
\right\}.
\]
Then
\begin{equation}
\label{eq:sec:clt 1}
\Delta_T^\xi(\theta)
=
\#\bigl(\Lambda_{\theta,\xi} \cap \Omega_T\bigr) + O(1),
\end{equation}
since the only integer points not captured by $\#(\Lambda_{\theta,\xi} \cap \Omega_T)$ correspond to integral vectors $(p,q)$ with $\|q\| \leq 1$, which contribute only a bounded error.

Now fix $T = e^N$ with $N \in \N$. Then
\[
\Omega_{e^N}
=
\bigsqcup_{s=0}^{N-1} b^{-s}\Omega_e,
\]
for $b = \left(\diag(e^{w_1},\dots,e^{w_m},e^{-1},\dots,e^{-1}),\,0\right),$ where $w_1,\dots ,w_m>0$ are weights as in Theorem \ref{thm: main thm}, i.e., $w_1+\dots+w_m=n.$

Hence
\begin{equation}
\label{eq:sec:clt 2}
\#\bigl(\Lambda_{\theta,\xi} \cap \Omega_{e^N}\bigr)
=
\sum_{s=0}^{N-1} \widehat{\ind}(b^s \Lambda_{\theta,\xi}),
\end{equation}
where $\ind$ denotes the indicator function of $\Omega_e$.

Combining \eqref{eq:sec:clt 1} and \eqref{eq:sec:clt 2}, we see that the proof of Theorem~\ref{thm: main thm} reduces to establishing a central limit theorem for the sums in \eqref{eq:sec:clt 2}. We formulate this precisely in the following proposition.

\begin{proposition}\label{thm:main_reduced}
Assume $m \ge 2$ and define
\begin{equation}\label{equ:F_N}
F_N
:=
\frac{1}{\sqrt{N}} \sum_{s=0}^{N-1}
\left(
\widehat{\ind} \circ b^s
-
\mu_{\Y_\xi}(\widehat{\ind} \circ b^s)
\right).
\end{equation}
Then for every $\eta \in \R$,
\[
\mu_{\Y_\xi}\bigl(\{y \in \Y_\xi : F_N(y) < \eta\}\bigr)
\longrightarrow
\mathcal{N}_{\sigma_{m,n}}(\eta)
\quad \text{as } N \to \infty,
\]
where $\sigma_{m,n}$ is as in Theorem~\ref{thm: main thm}.
\end{proposition}

\begin{proof}[Proof of Theorem~\ref{thm: main thm} assuming Proposition~\ref{thm:main_reduced}]
The argument is analogous to that of \cite[Thm.~1.2]{BG}, with \cite[Thm.~6.1]{BG} replaced by Proposition~\ref{thm:main_reduced} and \cite[Lem.~6.3]{BG} replaced by Lemma~\ref{lem:final}, and is therefore skipped.
\end{proof}
The remainder of this section is devoted to the proof of Proposition~\ref{thm:main_reduced}.

\subsection{Approximation by smooth compactly supported functions}

To prove Proposition~\ref{thm:main_reduced}, we employ the method of cumulants (Proposition~\ref{prop:Method of Cumulants}) together with the mixing estimate in Corollary~\ref{cor:Y_mul_mixing}. In order to apply Corollary~\ref{cor:Y_mul_mixing}, we first approximate $F_N$ by smooth, compactly supported functions. This is the goal of the present subsection.

To make this precise, we fix a family of non-negative functions $f_{\varepsilon} \in C_c^{\infty}(\R^{m+n})$ such that $\supp(f_\varepsilon)$ is contained in an $\varepsilon$-neighbourhood of $\Omega_e$ and
\begin{equation}\label{equ:chi_appro_f_epsilon}
    \ind \leq f_{\varepsilon} \leq 1, \,\, \|f_{\varepsilon}-\ind\|_{L^1(\mathbb R^{m+n})} \ll \varepsilon,\,\,\|f_{\varepsilon}-\ind\|_{L^2(\mathbb R^{m+n})} \ll \varepsilon^{1/2},\,\, \|f_{\varepsilon}\|_{C^k} \ll \varepsilon^{-k}.
\end{equation}

We then define
\begin{equation}\label{equ:F_N_L}
\tilde{F}_N^{(\varepsilon,L)}
:=
\frac{1}{\sqrt{N}} \sum_{s=M}^{N-1}
\left(
\widehat{f}_{\varepsilon}^{(L)} \circ b^s
-
\mu_{\Y_{\xi}}(\widehat{f}_{\varepsilon}^{(L)} \circ b^s)
\right),
\end{equation}
where the parameters $M = M(N) \to \infty$, $\varepsilon = \varepsilon(N) \to 0$, and $L = L(N) \to \infty$.

The following lemma shows that $\tilde{F}_N^{(\varepsilon,L)}$ and $F_N$ have the same limiting distribution as $N \to \infty$. 

\begin{lemma}\label{lem:F_N_F_N_epsilon_L} Let $F_N$ and $\tilde{F}_N^{(\varepsilon,L)}$ be defined as in \eqref{equ:F_N} and \eqref{equ:F_N_L}. Assume that the parameters $M,\varepsilon,L$ satisfy 
\begin{equation}\label{equ:M_epsilon_L}
    M=o(\log N), \,\, \varepsilon=o(N^{-1/2}), \text{ and } N=o(L^p) \text{ for some } p<m+n.
\end{equation}
Then
\[ \|F_N-\tilde{F}_N^{(\varepsilon,L)}\|_{L^1(\mathcal{Y}_{\xi})} \rightarrow 0 \text{ as } N \rightarrow \infty.\]
\end{lemma}
\begin{proof} Note that
\begin{align}\label{equ:tri}
   & \quad  \|F_N-\tilde{F}_N^{(\varepsilon,L)}\|_{L^1(\mathcal{Y}_{\xi})}  \nonumber \\ & \leq \frac{2}{\sqrt{N}} \sum_{s=0}^{M-1}\|\hat{\chi}\circ b^s \|_{L^1(\Y_{\xi})} + \frac{2}{\sqrt{N}} \sum_{s=M}^{N-1}\|(\hat{\chi}-\hat{f}_{\varepsilon}^{(L)})\circ b^s \|_{L^1(\Y_{\xi})} \nonumber \\ & \leq \frac{2M}{\sqrt{N}} \sup_{s\geq 0}\int_{\Y_{\xi}}|\hat{\chi}\circ b^s|\, d\mu_{\Y_{\xi}}+ \frac{2}{\sqrt{N}} \sum_{s=M}^{N-1}\|(\hat{\chi}-\hat{f}_{\varepsilon} )\circ b^s \|_{L^1(\Y_{\xi})} +\frac{2}{\sqrt{N}} \sum_{s=M}^{N-1}\|(\hat{f}_{\varepsilon}-\hat{f}_{\varepsilon}^{(L)})\circ b^s \|_{L^1(\Y_{\xi})}  
\end{align}
In view of Lemma~\ref{prop:hatf_L_1_2} and \eqref{equ:M_epsilon_L}, the first term of \eqref{equ:tri} goes to zero as $N\rightarrow \infty.$ For the second term, we use Lemma~\ref{prop:f_epsilon_chi_hat} and \eqref{equ:M_epsilon_L}, to get that 
\[\frac{2}{\sqrt{N}} \sum_{s=M}^{N-1}\|(\hat{\chi}-\hat{f}_{\varepsilon} )\circ b^s \|_{L^1(\Y_{\xi})} \ll N^{1/2}(\varepsilon+e^{-M}) \rightarrow 0\text{ as } N \rightarrow \infty.\]

For third term, observe that
 \begin{align}
& \quad \frac{2}{\sqrt{N}} \sum_{s=M}^{N-1}   \|(\hat{f}_{\varepsilon} -\hat{f}_{\varepsilon}^{(L)}) \circ b^s\|_{L^1(\Y_{\xi})} \nonumber \\ &\leq  \frac{2}{\sqrt{N}} \sum_{s=M}^{N-1}\int_{\{\alpha(b^s y) \geq c^{-1}L\}} |\hat{f_{\varepsilon}}(b^s y)| \, d\mu_{\Y_{\xi}}(y) \nonumber \\ &\leq  \frac{2}{\sqrt{N}} \sum_{s=M}^{N-1}\|\hat{f_{\varepsilon}}\circ b^s\|_{L^2(\Y_{\xi})} \,\,\mu_{\Y_{\xi}}(\{y \in \Y_{\xi}:\alpha(b^s y)\geq c^{-1}L\})^{1/2}, \,\, \text{by Cauchy-Schwarz inequaity} \nonumber \\ &\ll_p N^{1/2}L^{-p/2} \,\, \text{for some } p<m+n,   
 \end{align}
which tends to zero as $N\rightarrow \infty$, where in the last inequality we use Lemma~\ref{prop:f_epsilon_chi_hat} and~\ref{lem:Y_alpha_bound}.

\noindent  The desired conclusion now follows by combining the estimates established above.
\end{proof}

\medskip

\subsection{Estimating Cumulants}
\label{subsec:Estimate cum}
In this subsection, we compute the cumulants of $\tilde{F}_N^{(\varepsilon,L)}$ defined as in~\eqref{equ:F_N_L}. To do the same, fix $r \geq 3$. 

Using the computation in \cite[\S5.1 and \S6]{BG}, replacing Corollary \ref{cor:Y_mul_mixing} in place of \cite[Corollary~3.3]{BG} and using Remark \ref{rem:sobnorm_pro} in the relevant steps, the line below equation (6.10) of \cite{BG} yields
\begin{align}
\label{eq:cum evaluation}
    \left|\cum^{(r)}_{\mu_{{\Y}_\xi}}\left(\tilde{F}_N^{(\varepsilon,L)}\right) \right| &\ll N^{r/2} e^{-\delta \gamma} \mathcal{S}_k\left( \hat{f}_{\varepsilon}^{(L)}  \right)^r + N^{1-r/2} \gamma^{r-1} L^{(r-(m+n-1))^+} \|f_{\varepsilon}\|_{C^0}^{r} \nonumber\\
    &\ll N^{r/2} e^{-\delta \gamma} L^{r(k+1)} \|f_{\varepsilon}\|_{C^k}^r+ N^{1-r/2}\gamma^{r-1} L^{(r-(m+n-1))^+} \|f_{\varepsilon}\|_{C^0}^{r} \nonumber\\ & \ll  N^{r/2} e^{-\delta \gamma} L^{r(k+1)} \varepsilon^{-rk}+ N^{1-r/2}\gamma^{r-1} L^{(r-(m+n-1))^+},
\end{align}
where we use the notation $x^+=\max \{0,x\}$. The second and third inequalities above follow from Lemma \ref{lem:pro_f_hat_L} and \eqref{equ:chi_appro_f_epsilon}, respectively, provided that
\begin{equation}\label{equ:M_logL}
M \gg \log L \quad \text{and} \quad M \gg \gamma.
\end{equation}
Since $m\geq 2,$ we have
\[\frac{(r-(m+n-1))^+}{m+n}<r/2-1 \quad \forall \, r \geq 3.\]
Hence one can choose $q>\frac{1}{m+n}$ such that
\[q(r-(m+n-1))^+<r/2-1 \quad \forall \, r\geq 3.\]
Fix such a $q$, and let
\begin{align}
    \label{eq:choice of N}
    L=N^q.
\end{align}
Then we have
\begin{align}
\left|\cum^{(r)}_{\mu_{{\Y}_\xi}}\left(\tilde{F}_N^{(\varepsilon,L)}\right) \right|  & \ll  N^{r/2+qr(k+1)} e^{-\delta \gamma}  \varepsilon^{-rk}+ N^{q(r-(m+n-1))^+-(r/2-1)}\gamma^{r-1} . 
\end{align}
Further choosing $\gamma=c_r (\log N)$, where $c_r $ is sufficiently large, we get that
\[N^{q(r-(m+n-1))^+-(r/2-1)}\gamma^{r-1}  \rightarrow 0.\]
Hence 
\[\cum_{\mu_{\Y_{\xi}}}^{(r)} \left(\tilde{F}_N^{(\varepsilon,L)}\right) \rightarrow 0 \,\, \text{as } N \rightarrow \infty \text{ for } r\geq 3,\]
provided we have
\begin{equation}\label{equ:N_epsi_e_del}
   N^{r/2+qr(k+1)}    \varepsilon^{-rk}=o(e^{\delta \gamma}) .
\end{equation}
To make sure \eqref{equ:N_epsi_e_del} holds, we choose
\begin{align}
    \label{eq:choice of e}
    \varepsilon(N)=1/N^{\beta} \quad \text{with }  \beta<\frac{1}{k}\left(\frac{\delta c_r}{r}-\frac{1}{2}-q(k+1)\right), 
    \end{align}
 provided $c_r>\frac{r}{\delta}(k+1)(q+1/2)$  is sufficiently large.
Finally choose 
\begin{align}
    \label{eq:choice of M}
    M(N)=(\log N) (\log \log N)
\end{align}
so that \eqref{equ:M_logL} holds.
With the above choice of parameters $\varepsilon(N),L(N),\text{ and }M(N)$, we have proved that the $r$-th cumulants vanish as $N \rightarrow \infty$ for $r \geq 3$.

\subsection{Estimating variance}
\label{subsec:Estimate var}
In this section, we compute the variance of $\tilde{F}_N^{(\varepsilon,L)}$ with $L(N)$, $\varepsilon(N)$, and $M(N)$ chosen as in~\eqref{eq:choice of N}, \eqref{eq:choice of e}, and \eqref{eq:choice of M} respectively. The computation of variance is also similar to \cite[\S6]{BG}, we point out the differences. 

By analogous computation to \cite[Theorem 6.1]{BG} (look at line after $(6.16)$ of \cite{BG}), we get
\begin{equation}
\label{aff eq: main variance}
\left\|\Tilde{F}^{(\e,L)}_N\right\|_{L^2(\Y_{\xi})}^2 =\Theta^{(\e)}_\infty(0)+2\sum_{s=1}^{K-1}\Theta^{(\e)}_\infty (s)+o(1),\end{equation}
where $$\Theta^{(\e)}_\infty(s):=\int_{\tilde{\X}} (\hat f_\e\circ b^s)\hat f_\e\, d\mu_{\tilde{\X}} -\mu_{\tilde{\X}}(\hat f_\e)^2,$$
and $K=K(N):=c_1 (\log N)$ is a parameter ($c_1$ is sufficiently large).

\noindent In view of Proposition \ref{prop:roger}, we have
\begin{align*}
\Theta^{(\e)}_\infty (s) &= \int_{\tilde{\X}} (\hat f_\e\circ b^s)\hat f_\e\, d\mu_{\tilde{\X}} -\mu_{\tilde{\X}}(\hat f_\e)^2 \\
&=  \frac{1}{2} \left(\int_{\tilde{\X}}(\hat{f}_\e \circ b^s + \hat{f}_\e)^2\, d\mu_{\tilde{\X}} - 2\int_{\tilde{\X}}( \hat{f}_\e)^2\, d\mu_{\tilde{\X}}\right)  -\mu_{\tilde{\X}}(\hat f_\e)^2  \\
    &= \frac{1}{2} \left(\left( \int_{\R^d}(f_\e \circ b^s + f_\e)\,\, dx \right)^2 + \int_{\R^d}(f_\e \circ b^s + f_\e )^2 \,\, dx \right)\\&- \left( \int_{\R^d} f_\e \,\,  dx \right)^2 - \int_{\R^d}( f_\e )^2 \,\,  dx - \left( \int_{\R^d} f_\e \,\,  dx \right)^2 \\
    &= \int_{\R^{m+n}} f_\e(b^s x) f_\e(x)\,  dx.
\end{align*}
We claim that
\begin{equation}\label{equ:Theta_infty_epsilon_minus_Theta}
    | \Theta^{(\e)}_\infty (s) - \Theta_{\infty}(s)| \ll \varepsilon^{1/2},
\end{equation}
where \[\Theta_{\infty}(s)= \int_{\R^{m+n}} \ind(b^s x) \ind(x) \,  dx. \]

\noindent Indeed, by using the Cauchy–Schwarz inequality and \eqref{equ:chi_appro_f_epsilon}, we get
\begin{align*}
    | \Theta^{(\e)}_\infty (s) - \Theta_{\infty}(s)| &\leq \int_{\R^{m+n}} |f_\e(b^s x)- \ind(b^s x )| |f_\e(x)|\,  dx + \int_{\R^{m+n}} |\ind(b^s x)| |f_\e(x)- \ind(x)|\,  dx \\
    &\leq \|f_\e - \ind\|_{L^2(\R^{m+n})} \|f_\e\|_{L^2(\R^{m+n})} + \|\ind\|_{L^2(\R^{m+n})} \|f_\e- \ind\|_{L^2(\R^{m+n})} \\
    &\ll \e^{1/2}.
\end{align*}
Since $\varepsilon=\varepsilon(N) \rightarrow 0$ and $K=K(N) \rightarrow \infty$ as $N \rightarrow \infty,$ the estimate \eqref{equ:Theta_infty_epsilon_minus_Theta} and \eqref{aff eq: main variance} gives
 \[\left\|\Tilde{F}^{(\e,L)}_N\right\|_{L^2(\Y_{\xi})}^2 \rightarrow \Theta_\infty(0)+2\sum_{s=1}^{\infty}\Theta_\infty (s) \quad \text{as }N \rightarrow \infty.\]
Now the direct computation shows that 
\[
\Theta_{\infty}(s)=0 \quad \text{for all } s \in \mathbb N,
\]
and  $ $
   \[\Theta_{\infty}(0)= \sigma_{m,n}^2,\]
   where $ \sigma$ is as in Theorem~\ref{thm: main thm}. Hence, we get that
   \begin{align}
       \label{eq;var comp}
       \left\|\Tilde{F}^{(\e,L)}_N\right\|_{L^2(\Y_{\xi})}^2 \rightarrow \sigma_{m,n}^2.
   \end{align}

   \subsection{Proof of Proposition~\ref{thm:main_reduced}}
   \begin{proof}
       The proposition follows directly using the criteria given in Proposition~\ref{prop:Method of Cumulants} along with Lemma~\ref{lem:F_N_F_N_epsilon_L} and the calculations of  \S~\ref{subsec:Estimate cum} and \S~\ref{subsec:Estimate var}.
   \end{proof}
   
   %equals the measure of the set $\Omega_e$, defined as in \eqref{eq:def Omega}. Hence, $\Theta_{\infty}(0)= 2^{m+1} \vartheta_1\vartheta_2 \ldots \vartheta_m \omega_n \text{ with } \omega_n := \int_{S^{n-1}}  d\bar{z}$. 

 %Combining the above statements, we get that since $\e(M) \rightarrow 0$ as $M \rightarrow \infty$, so \begin{align*}\Theta^{(\e)}_\infty(0)+2\sum_{s=1}^{K(M)-1}\Theta^{(\e)}_\infty (s) \rightarrow 2^m \vartheta_1\vartheta_2 \ldots \vartheta_m \omega_n,\end{align*}which proves that \begin{equation}\label{aff variance cal 1}\Theta^{(\e)}_\infty(0)+2\sum_{s=1}^{K-1}\Theta^{(\e)}_\infty (s) \rightarrow 2^m \vartheta_1\vartheta_2 \ldots \vartheta_m \omega_n\end{equation} as $M \rightarrow \infty$.Now, resuming the steps in proof of Theorem 6.1 of \cite{BG} will give us the proof.

\appendix

\section{Auxiliary bounds and estimates}

In this section, we prove three lemmas used in the proof of Theorem~\ref{thm: main thm}. The arguments follow the same general strategy as in \cite{BG}, with minor modifications to accommodate our setting. For completeness, we include the proofs here.

\begin{lemma}\label{lem:final}
    \[\sum_{s=0}^{N-1}\int_{\Y_{\xi}} \hat{\ind}(b^sy)\,\, d\mu_{\Y_{\xi}}(y)=C_{m,n} N+O(1),\]
    where $C_{m,n}=2^m \vartheta_1\dots\vartheta_m \omega_m$ with $\omega_n=\int_{\mathbb S^{n-1}}\|\bar{z}\|^{-n}\,\,d\bar{z}$
\end{lemma}
\begin{proof}
    First note that
    \[\sum_{s=0}^{N-1}\int_{\Y_{\xi}} \hat{\ind}(b^sy)\,\, d\mu_{\Y_{\xi}}(y)=\int_{\Y_{\xi}}\hat{\phi}_N(y) \,\, d\mu_{\Y_{\xi}}(y),\]
 where $\phi_N$ denotes the characteristic function of the set
 \[\left\{ (\bar{x},\bar{y})\in \mathbb R^{m+n}:1\leq \|\bar{y}\|<e^N,\,\, |x_i|<\vartheta_i \|\bar{y}\|^{-w_i},\,\,i=1,\dots,m\right\}.\]
Now
 \begin{eqnarray*}
     \int_{\Y_{\xi}}\hat{\phi}_N(y) \,\, d\mu_{\Y_{\xi}}(y) &=& \sum_{1 \leq \|\bar{q}\|<e^N} \prod_{i=1}^m \sum_{p_i \in \Z} \int_{[0,1]^n} \ind'_{\vartheta \|\bar{q}\|^{-w_i}}(p_i+\langle \bar{\theta}_i,\bar{q} \rangle +\xi_i)\,\, d\bar{\theta}_i,
 \end{eqnarray*}
 where $\ind'_{M}$ denotes the characteristic function of $[-M,M]$.
 
We claim that for fixed $\xi_0 \in \mathbb R$
\begin{equation}\label{equ:sum_p_int_chi}
    \sum_{p \in \Z} \int_{[0,1]^n} \ind'_{\vartheta \|\bar{q}\|^{-w_i}}(p+\langle \bar{\theta},\bar{q} \rangle +\xi_0)\,\, d\bar{\theta}= 2 \vartheta_i \|\bar{q}\|^{-w_i}.
\end{equation}
For proving this, we consider a general compactly supported bounded measurable function $\varphi$ on $\mathbb R$, the function $\phi(x)=\varphi(x_1+\xi_0)$ on $\mathbb R^m,$ and the function $\tilde{\phi}(x)=\sum_{p\in \mathbb Z}\varphi(p+x_1+\xi_0)$ on the torus $\mathbb  R^m/\mathbb Z^m.$ Without loss of generality we can assume $q_1 \neq 0$ and hence the following linear map $S: \mathbb R^m \rightarrow \mathbb R^m$ given by 
\[S(\bar{\theta})=(\langle \bar{\theta},\bar{q}\ \rangle,\theta_2,\dots,\theta_m)\,\, \text{ for } \bar{\theta} \in \R^m \]
is non-degenerate and induces a linear epimorphism of the torus $\mathbb R^m/\mathbb Z^m.$ Now, since $S$ preserves the Lebesgue probability measure $\lambda$ on $\R^m/\Z^m$ and the Lebesgue measure on $\R/\Z$ is translation invariant, we have 
\begin{eqnarray*}
    \sum_{p \in \Z} \int_{[0,1]^m} \varphi (p+\langle \bar{\theta},\bar{q}\rangle +\xi_0)= \int_{\R^m/\Z^m} \tilde{\phi}(Sx) d\lambda(x)  =    \int_{\R^m/\Z^m} \tilde{\phi}(x) d\lambda(x)  = \int_{\R} \varphi(x_1)\,\, dx_1.
\end{eqnarray*}
By taking, $\phi=\ind'_{\vartheta \|\bar{q}\|^{-w_i}},$ we have our desired claim \eqref{equ:sum_p_int_chi}. Now using \eqref{equ:sum_p_int_chi}, we get 
\[\int_{\Y_{\xi}}\hat{\phi}_N(y) \,\, d\mu_{\Y}(y) = 2^m \left(\prod_{i=1}^m \vartheta \right)\sum_{1 \leq \|\bar{q}\|<e^N} \|\bar{q}\|^{-n}.  \]
Since $\|\bar{y}_2\|^{-n}=\|\bar{y}_1\|^{-n}+O(1)$ whenever $\|\bar{y}_1-\bar{y}_2\| \ll 1,$ we get that
\[\sum_{1 \leq \|\bar{q}\|<e^N} \|\bar{q}\|^{-n}=\int_{1 \leq \|\bar{y}\|<e^N} \|\bar{y}\|^{-n} \, d\bar{y} +O(1),\]
and further expressing the integral in polar coordinates, we have
\[\int_{1 \leq \|\bar{y}\|<e^N} \|\bar{y}\|^{-n} \, d\bar{y}=\int_{\mathbb S^{n-1}}\int_{\|\bar{z}\|^{-1}}^{e^N\|\bar{z}\|^{-1}} \|r\bar{z}\|^{-n}r^{n-1}\,\,dr\,d\bar{z}=\omega_n N +O(1).\]
\end{proof}

\begin{lemma}\label{prop:hatf_L_1_2}
   Let $m \geq 2.$ Suppose $f:\mathbb R^{m+n} \rightarrow \mathbb R$ is a bounded measurable function with compact support contained in the open set $\{(x_1,\dots,x_{m+n}) \in \mathbb R^{m+n}: (x_{m+1},\dots,x_{m+n}) \neq 0\}.$ Then for $i=1,2$, we have
    \[\sup_{s\geq 0}\|\hat{f} \circ b^s\|_{L^{i}(\Y_{\xi})} < \infty.\]
\end{lemma}
\begin{proof}
First, observe that using Cauchy--Schwarz inequality, it suffices to prove the $L^2$ boundedness. To this end, we closely follow the strategy of \cite[Proposition 4.8]{BG}.
From the assumption on the function $f$, there exist constants $0<v_1 <v_2$ and $\vartheta >0$ such that  the support of $f$ is contained in the set 
\begin{equation}
  \Omega_{v_1,v_2,\vartheta}:= \{(\bar{x},\bar{y}) \in \mathbb R^{m+n}:v_1 \leq\|\bar{y}\|\leq v_2,\,\, |x_i|\leq \vartheta  \|\bar{y}\|^{-w_i},\,\, i=1,\dots,m\},
\end{equation}
and without loss of generality, we may assume that $f$ is the characteristic function of $\Omega_{v_1,v_2,\vartheta}.$ Recall that $\xi \in \R^m$ is fixed, and
 \[ \mathcal{Y}_{\xi}=\left\{\Lambda_{\theta,\xi}:\theta=(\theta_{ij})_{m \times n} \in  \Mat \right\}=\left\{\Lambda_{\theta,\xi}:\theta  \in  \mathrm{M}_{m \times n}([0,1]) \right\},\]
where  
\[\Lambda_{\theta,\xi}=\left\{\left(p_1+\sum_{j=1}^{n}\theta_{1j}q_j+\xi_1,\dots,p_m+\sum_{j=1}^{n}\theta_{mj}q_j+\xi_m,q_1,\dots,q_n\right): p_1,\dots,p_m,q_1,\dots,q_n \in \mathbb Z\right\}.\]
We set $\bar{\theta_i}:=(\theta_{i1},\dots,\theta_{in})$ and denote the characteristic function of $[-M,M]$ by $\ind'_{M}.$ Then 
\begin{eqnarray}
    \hat{f}(b^s \Lambda_{\theta,\xi}) &=&\sum_{(\bar{p},\bar{q}) \in \mathbb Z^{m+n} \setminus \{0\}} f(e^{w_1 s}(p_1+\langle \bar{\theta_1},\bar{q}\rangle+\xi_1),\dots,e^{w_m s}(p_m+\langle \bar{\theta_m},\bar{q}\rangle+\xi_m),e^{-s}\bar{q}) \nonumber\\ &=& \sum_{v_1 e^s \leq \|\bar{q}\|\leq v_2 e^s} \sum_{\bar{p}\in \mathbb Z^{m}} \prod_{i=1}^{m}\ind'_{\vartheta \|\bar{q}\|^{-w_i}}(p_i+\langle \bar{\theta_i},\bar{q}\rangle+\xi_i) \nonumber \\ & =& \sum_{v_1 e^s \leq \|\bar{q}\|\leq v_2 e^s} \prod_{i=1}^{m} \sum_{p_i \in \mathbb Z}  \ind'_{\vartheta \|\bar{q}\|^{-w_i}}(p_i+\langle \bar{\theta_i},\bar{q}\rangle+\xi_i), \nonumber
\end{eqnarray}
which implies that
\begin{eqnarray}
\|f \circ b^s\|_{L^2(\mathcal{Y}_{\xi})}^2  & =&  \sum_{v_1 e^s \leq \|\bar{q}\|,\|\ell\| \leq v_2 e^s} \prod_{i=1}^{m}\left( \sum_{p_i,r_i \in \mathbb Z}   \int_{[0,1]^n} \,\ind'_{\vartheta \|\bar{q}\|^{-w_i}}(p_i+\langle \bar{\theta_i},\bar{q}\rangle+\xi_i) \ind'_{\vartheta \|\bar{\ell}\|^{-w_i}}(r_i+\langle \bar{\theta_i},\bar{\ell}\rangle+\xi_i) d\bar{\theta_i} \right).\nonumber
\end{eqnarray}
Hence for fixed $\bar{q},\bar{\ell} \in \mathbb Z^n,$ and $v_0 \in [0,1],$ it is enough to estimate the following integral:
\[I_i(\bar{q},\bar{\ell})=\sum_{p,r \in \mathbb Z}   \int_{[0,1]^n} \ind'_{\vartheta \|\bar{q}\|^{-w_i}}(p+\langle \bar{u},\bar{q}\rangle+v_0) \, \ind'_{\vartheta \|\bar{\ell}\|^{-w_i}}(r+\langle \bar{u},\bar{\ell}\rangle+v_0) \,d\bar{u}. \]
Estimating $I_i(\bar{q},\bar{\ell})$ is divided into two cases based on whether $\bar{q}$ and $\bar{\ell}$ are linearly independent. First, we consider the case when $\overline{q}$ and $\overline{\ell}$ are linearly independent. Then there exist indices $j,k=1,\ldots,n$ such that $q_j\ell_k-q_k\ell_j\ne 0.$ Let us consider the function $\phi$ on $\mathbb R^2$ defined by
$\phi(x_1,x_2)=\chi^{(i)}_{\overline{q}}(x_1)\chi^{(i)}_{\overline{\ell}}(x_2)$
and the periodized function $\bar\phi$ on $\mathbb R^2/\mathbb Z^2$ defined by 
$\bar\phi(x)=\sum_{z\in \mathbb Z^2} \phi(z+x)$.
 Set,
\[
\omega:=\sum_{\zeta\ne j,k} q_\zeta u_\zeta +v_0 \quad \text{and} \quad \rho:=\sum_{\zeta\ne j,k} \ell_\zeta u_\zeta +v_0.
\] 
Let $S$ the affine map given by
$$
S:(x_1,x_2)\mapsto (q_jx_1 +q_kx_2+\omega, \ell_j x_1+\ell_k x_2+\rho),
$$
which induces an affine endomorphism of the torus $\mathbb R^2/\mathbb Z^2$. 

Suppose $\mu$ be the Lebesgue probability measure on the torus. Then
$$
\sum_{p,r\in\mathbb Z}\int_{[0,1]^2}\chi^{(i)}_{\overline{q}}\left(p+\left<\overline{u},\overline{q}\right>+v_0\right)\chi^{(i)}_{\overline{\ell}}\left(r+\left<\overline{u},\overline{\ell}\right>+v_0\right)\, du_jdu_k=\int_{\mathbb R^2/\mathbb Z^2} \bar{\phi}(Sx)\, d\mu(x).
$$
Since  $S$ preserves $\mu$, we have
$$
\int_{\mathbb R^2/\mathbb Z^2} \bar{\phi}(Sx)\, d\mu(x)=\int_{\mathbb R^2/\mathbb Z^2} \bar{\phi}(x)\, d\mu(x)
=\int_{\mathbb R^2} \phi(x)\, dx=4\vartheta^2\, \|\overline{q}\|^{-w_i} \|\overline{\ell}\|^{-w_i}.
$$
Therefore, in this case, we get that
\begin{equation}
\label{eq:II1}
I_i(\overline{q},\overline{\ell})\ll \|\overline{q}\|^{-w_i} \|\overline{\ell}\|^{-w_i}.
\end{equation}

Let us now consider the second case when $\overline{q}$ and $\overline{\ell}$ are linearly dependent. After possibly reordering the indices, we may assume that
\begin{equation}\label{eq:q}
|q_1|=\max(|q_1|,\ldots,|q_n|,|\ell_1|,\ldots, |\ell_n|).
\end{equation}
This implies that $q_1 \neq 0$. Since $\bar{q}$ and $\bar{\ell}$ are linearly dependent, it follows that $\ell_1 \neq 0$ as well. Hence, we may introduce new variables as follows:
 $$
v_1=\sum_{j=1}^n (q_j/q_1)u_j = \sum_{j=1}^n (\ell_j/\ell_1)u_j \quad v_2=u_2,\ldots, v_n=u_n,
$$
and so
\begin{equation}\label{eq:I_i_less_J_i}
    I_i(\overline{q},\overline{\ell})\le J_i(\overline{q},\overline{\ell}),
\end{equation}
where
$$
J_i(\overline{q},\overline{\ell}):=
\sum_{p,r\in\mathbb{Z}}
\int_{-n}^n\chi^{(i)}_{\overline{q}}\left(p+q_1 v_1+v_0\right)\,\chi^{(i)}_{\overline{\ell}}\left(r+\ell_1v_1+v_0\right)\, dv_1.
$$
Note that the last integral is non-zero only when $|p|\ll |q_1|$ and $|r|\ll |\ell_1|$, where the implied constants depend on $n$ and $v_0$.
Write $q_1=q'd$ and $\ell_1=\ell'd$ where $d=\gcd(q_1,\ell_1)$.
Then $q_1r-\ell_1 p=j d$ for some $j\in \mathbb{Z}$. Observe that when $j$ is fixed, then the integers $p$ and $r$ satisfy the 
equation $q'r-\ell'p=j$. Since $\gcd(q',\ell')=1$,
all the solutions of the above equation are given by $p=p_0+kq'$, $r=r_0+k\ell'$
for $k\in \mathbb{Z}$. It follows that the number of such solutions
satisfying $|p|\ll |q_1|$ and $|r|\ll |\ell_1|$ is at most $O(d)$.
 Accordingly, we decompose
$$
J_i(\overline{q},\overline{\ell})=J^{(1)}_i(\overline{q},\overline{\ell})+J^{(2)}_i(\overline{q},\overline{\ell}),
$$
where $J_i^{(1)}$ is the sum   over those $p,r$ with $q_1r-\ell_1 p\ne 0$,
and $J_i^{(2)}$ is the sum over those $p,r$ with $q_1r-\ell_1 p=0$.\\

Applying the linear change of variables $v_1'=v_1+\frac{p+v_0}{q_1}$ (and relabeling $v_1'$ again by $v_1$), we obtain  
\begin{align*}
J^{(1)}_i(\overline{q},\overline{\ell})
&=\sum_{p,r:\, q_1r-\ell_1 p\ne 0}
\int_{-n+(p+v_0)/q_1}^{n+(p+v_0)/q_1}\chi^{(i)}_{\overline{q}}(q_1v_1)\,\chi^{(i)}_{\overline{\ell}}\left((q_1r-\ell_1 p)/q_1+\ell_1v_1+v_0(1-\frac{\ell_1}{q_1})\right)\, dv_1\\
&\ll d \sum_{j\in \mathbb{Z}\backslash\{0\}} 
\int_{-\infty}^{\infty}\chi^{(i)}_{\overline{q}}(q_1 v_1)\,\chi^{(i)}_{\overline{\ell}}\left(jd/q_1+\ell_1 v_1+v_0(1-\ell_1/q_1)\right)\, dv_1.
\end{align*}
Define
$$
\rho_i(x):=  \int_{-\infty}^{\infty} \chi^{(i)}_{\overline{q}}(q_1 v_1)\,\chi^{(i)}_{\overline{\ell}}\left(xd/q_1+\ell_1 v_1+v_0(1-\ell_1/q_1)\right)\, dv_1.
$$
The integrand is the indicator function of the intersection of the intervals
$$
\left[-\vartheta\,|q_1|^{-1}\|\overline{q}\|^{-w_i},\vartheta\, |q_1|^{-1}\|\overline{q}\|^{-w_i}\right]
$$
and 
$$
\left[-xd/(q_1\ell_1)-\vartheta\,|\ell_1|^{-1}\|\overline{\ell}\|^{-w_i}-\frac{v_0}{\ell_1}(1-\ell_1/q_1),-xd/(q_1\ell_1)+\vartheta\,|\ell_1|^{-1}\|\overline{\ell}\|^{-w_i}-\frac{v_0}{\ell_1}(1-\ell_1/q_1)\right].
$$
%Suppose, for instance, that $q_1\ell_1>0$. 
It follows that $\rho_i$ is non-increasing when $x\ge 0$, and non-decreasing when $x\le 0$.
 Consequently,
$$
\sum_{j\in \mathbb{Z}\backslash \{0\}} \rho_i(j)
\le \int_{-\infty}^{\infty} \rho_i(x)\, dx.
$$
Moreover,
 \begin{eqnarray}
    \int_{-\infty}^{\infty} \rho_i(x)\, dx &=&
\int_{-\infty}^{\infty} \chi^{(i)}_{\overline{q}}(q_1 v_1) \left(\int_{-\infty}^{\infty}\chi^{(i)}_{\overline{\ell}}\left(xd/q_1+\ell_1 v_1+v_0(1-\ell_1/q_1)\right)\, dx\right)\,dv_1 \nonumber \\ &=& \int_{-\infty}^{\infty} \chi^{(i)}_{\overline{q}}(q_1 v_1) \left(\frac{2q_1\vartheta \|\ell\|^{-w_i}}{d  }\right) \, dv_1
\ll d^{-1} \|\overline{q}\|^{-w_i} \|\overline{\ell}\|^{-w_i}.
 \end{eqnarray}
Therefore,
\begin{align*}
J_i^{(1)}(\overline{q},\overline{\ell})\ll d \sum_{j\in \mathbb{Z}\backslash \{0\}} \rho_i(j)
\ll \|\overline{q}\|^{-w_i}\|\overline{\ell}\|^{-w_i}.
\end{align*}
%This argument also applies when $q_1\ell_1<0$ and implies the same bound.
We now estimate $J^{(2)}_i(\overline{q},\overline{\ell})$.
Let $c_0:=\min\{\|\overline{q}\|: \overline{q} \in \mathbb Z^{n}\backslash \{0\}\}$.
Denote by $N(q_1,\ell_1)$ the number of solutions $(p,r)$ of the equation
$$
q_1r-\ell_1 p= 0\quad\hbox{ with $|p|\le (c_0^{-n}\vartheta+n+|v_0|)|q_1|$,} % and $|r|\le (\vartheta+n)|\ell_1|$}.
$$
By change of variable $v_1'=v_1+p/q_1=v_1+r/\ell_1$, we obtain
\begin{align*}
	J^{(2)}_i(\overline{q},\overline{\ell})
	&=\sum_{p,r:\, q_1r-\ell_1 p= 0}
	\int_{-n+p/q_1}^{n+p/q_1}\chi^{(i)}_{\overline{q}}(q_1v_1+v_0)\,\chi^{(i)}_{\overline{\ell}}\left(\ell_1v_1+v_0\right)\, dv_1\\
	&\le N(q_1,\ell_1) \int_{-\infty}^{\infty}\chi^{(i)}_{\overline{q}}(q_1v_1+v_0)\,\chi^{(i)}_{\overline{\ell}}\left(\ell_1v_1+v_0\right)\, dv_1\\
	&\ll N(q_1,\ell_1)|q_1|^{-1}\|\overline{q}\|^{-w_i}\ll N(q_1,\ell_1)\max\left(\|\overline{q}\|,\|\overline{\ell}\|\right)^{-(1+w_i)},
\end{align*}
where we used that $q_1$ is chosen according to \eqref{eq:q}.
Combining the bounds for $J_i^{(1)}(\overline{q},\overline{\ell})$
and $J^{(2)}_i(\overline{q},\overline{\ell})$, we conclude that when 
$\overline{q}$ and $\overline{\ell}$ are 
linearly dependent,
\begin{equation}
\label{eq:II2}
J_i(\overline{q},\overline{\ell})\ll \|\overline{q}\|^{-w_i}\|\overline{\ell}\|^{-w_i}+
N(q_1,\ell_1)\max\left(\|\overline{q}\|,\|\overline{\ell}\|\right)^{-(1+w_i)},
\end{equation}
where $q_1$ is chosen according to \eqref{eq:q}.

 Now combining \eqref{eq:II1}, \eqref{eq:I_i_less_J_i} and \eqref{eq:II2} with \cite[Lemma 4.9]{BG} gives our desired $L^2$ bound.
\end{proof}

\begin{lemma}\label{prop:f_epsilon_chi_hat}
    Given any $s\geq 0,$
    \[\int_{\Y_{\xi}} \left|\hat{f_{\varepsilon}} \circ b^s -\hat{\ind} \circ b^s\right| d\mu_{\Y_{\xi}} \ll \varepsilon + e^{-s}.\]
\end{lemma}
\begin{proof}
It is easy to see that there exists $\vartheta_{i}(\varepsilon)> \vartheta_i$ such that $\vartheta_i(\varepsilon)= \vartheta_i + O(\varepsilon)$ and $f_{\varepsilon} \leq \ind_{\varepsilon}$, where $\ind_{\varepsilon}$ denotes the characteristic function of the set 
        \begin{equation}
            \{(\bar{x},\bar{y}) \in \R^{m+n}: 1- \varepsilon \leq \|\bar{y}\| \leq e +\varepsilon, \,\,|x_i| < \vartheta_{i}(\varepsilon) \|\bar{y}\|^{-w_i} \text{  for } i= 1,\ldots, m\}.
        \end{equation}
        Now we have
        \begin{equation*}
            |\hat{f_{\varepsilon}}(b^s \Lambda) - \hat{\ind}(b^s\Lambda)| = \sum_{v \in \Lambda \setminus \{0\}} (f_{\varepsilon}(b^s v)- \ind(b^s v)) \leq \sum_{v \in \Lambda \setminus \{0\}} (\ind_{\varepsilon}(b^s v)- \ind(b^s v)) .
        \end{equation*}
        It is easy to note that $\ind_{\varepsilon}- \ind$ is bounded by the sum $\ind_{1,\varepsilon} + \ind_{2,\varepsilon} + \ind_{3,\varepsilon}$ of the characteristic functions of the sets
        \begin{eqnarray*}
        &\{(\bar{x},\bar{y}) \in \R^{m+n} : 1-\varepsilon \leq \|\bar{y}\| \leq 1,\,\,  |x_i| < \vartheta_{i}(\varepsilon) \|\bar{y}\|^{-w_i} \text{  for } i= 1,\ldots, m\}, \\
        &\{(\bar{x},\bar{y}) \in \R^{m+n} : e \leq \|\bar{y}\| \leq  e+ \varepsilon,\,\,  |x_i| < \vartheta_{i}(\varepsilon) \|\bar{y}\|^{-w_i} \text{  for } i= 1,\ldots, m\}, \\
        &\{(\bar{x},\bar{y}) \in \R^{m+n} : 1 \leq \|\bar{y}\| \leq  e, \,\, |x_i| < \vartheta_{i}(\varepsilon) \|\bar{y}\|^{-w_i} \text{  for } i= 1,\ldots, m, |x_j|\geq \vartheta_j \|\bar{y}\|^{-w_j} \text{  for some } j\}
        \end{eqnarray*}
        respectively. Therefore, we have
        \begin{equation*}
            \hat{f}_{\varepsilon}(b^s \Lambda) - \hat{\ind}(b^s\Lambda) \leq \hat{\ind}_{1,\varepsilon}(b^s \Lambda) + \hat{\ind}_{2,\varepsilon}(b^s \Lambda) + \hat{\ind}_{3,\varepsilon}(b^s \Lambda).
        \end{equation*}
        Therefore, it remains to show that 
        \begin{equation*}
            \int_{\Y_{\xi}} (\hat{\ind}_{i, \varepsilon} \circ b^s) \, d\mu_{\Y_{\xi}} \ll \varepsilon + e^{-s} \quad \text{for } i=1,2,3.
        \end{equation*}
Let $\chi'_M$ denote the characteristic function of $[-M,M]$. Now we compute that 
\begin{align*}
\int_{\Y_{\xi}} (\hat{\ind}_{1, \varepsilon} \circ b^s) \, d\mu_{\Y_{\xi}} &= \underset{(1- \e)e^s \leq \|\bar{q}\| \leq e^s}{ \sum} \prod_{i=1}^{m}    \underset{{p_i} \in \Z}{ \sum} \int_{[0,1]^n}  \ind_{\vartheta_i(\e)\|\bar{q}\|^{-w_i}}'(p_i+ \langle \bar{\theta}_i, \bar{q} \rangle + \xi_i) \,\, d\bar{\theta_i}  \nonumber \\
            &\ll \underset{(1- \e)e^s \leq \|\bar{q}\| \leq e^s}{ \sum} \prod_{i=1}^{m}     \left(2\vartheta(\varepsilon)\|\bar{q}\|^{-w_i} (\max_{k}|q_k|)^{-1} \|\bar{q}\|\right)  \\
             &\ll \underset{(1- \e)e^s \leq \|\bar{q}\| \leq e^s}{ \sum}  \|\bar{q}\|^{-n} \nonumber \\
             &\ll e^{-ns} \, \#\{ \bar{q} \in \Z^n: (1-\varepsilon)e^s \leq \|\bar{q}\| \leq e^s \}.
         \end{align*}
         We can estimate the number of integral points in the region $\{ (1-\varepsilon)e^s \leq \|\bar{q}\| \leq e^s \}$ in terms of its volume. In other words, there exists $r>0$ (depending only on the norm) such that 
        \begin{equation*}
           \#\{ \bar{q} \in \Z^n: (1-\varepsilon)e^s \leq \|\bar{q}\| \leq e^s \} \ll  |\{ \bar{y} \in \R^n: (1-\varepsilon)e^s -r \leq \|\bar{y}\| \leq e^s +r \}|.
        \end{equation*}
        Therefore, finally we have 
        \begin{eqnarray*}
            \int_{\Y_{\xi}} (\hat{\ind}_{1, \varepsilon} \circ b^s) \, d\mu_{\Y_{\xi}} &\ll& e^{-ns} ((e^s +r)^n- ((1-\varepsilon)e^s-r)^n) \\
            &=& (1 +re^{-s})^n- ((1-\varepsilon)-re^{-s})^n \\
            &\ll& \varepsilon + e^{-s}.
        \end{eqnarray*}
        Note that the integral estimate of $\hat{\ind}_{2,\varepsilon} \circ b^s$ is similar to the above; hence, we omit the details.\\ 

       Similar calculations of the integral over $\hat{\ind}_{3,\varepsilon} \circ b^s$ gives
        \begin{align*}
           \int_{\Y_{\xi}} (\hat{\ind}_{3, \varepsilon} \circ b^s) \, d\mu_{\Y_{\xi}} &= \underset{e^s \leq \|\bar{q}\| \leq e^{s+1}}{ \sum}  2^m \left(\prod_{i=1}^m\vartheta_{i}(\varepsilon) - \prod_{i=1}^m \vartheta_i \right)   \|\bar{q}\|^{-n} \\
           &\ll \varepsilon \left(\underset{e^s \leq \|\bar{q}\| \leq e^{s+1}}{ \sum}  \|\bar{q}\|^{-n} \right) \\
           &\ll \varepsilon.
        \end{align*}
        This completes the proof of the lemma.
\end{proof}

\bibliographystyle{alpha}
\bibliography{biblio}

\end{document}